\documentclass[12pt,a4paper,reqno]{amsart}
\usepackage{amssymb,a4wide}
\usepackage{tensor} 
\usepackage{hyperref}

\newcommand{\R}{\mathbb{R}} \newcommand{\N}{\mathbb{N}}              
 \newcommand{\C}{\mathbb{C}}      
\newcommand{\Rn}{{\R}^n}

\newcommand{\ca}{{\mathcal A}}
\newcommand{\cb}{{\mathcal B}}
\newcommand{\cd}{{\mathcal D}}
\newcommand{\ce}{{\mathcal E}}
\newcommand{\cf}{{\mathcal F}}
\newcommand{\cfi}{{\mathcal F}^{-1}}
\newcommand{\cs}{{\mathcal S}}

\newcommand{\lpvec}{L_{\vec{p}}}
\newcommand{\lloc}{L_{1,\mathrm{loc}}}

\newcommand{\supp}{\operatorname{supp}}
\newcommand{\dist}{\operatorname{dist}}

\newtheorem{thm}{Theorem}
\newtheorem*{theorem*}{{\bf Theorem\/}}
\newtheorem{defn}{Definition}
\newtheorem{cor}{Corollary}
\newtheorem{lem}{Lemma}
\newtheorem{prop}{Proposition}
\newtheorem{rem}{Remark}

\title[Mixed Norms and Diffeomorphisms]{%
Anisotropic, Mixed-Norm Lizorkin--Triebel Spaces and Diffeomorphic Maps}
\author[Johnsen, Munch Hansen, Sickel]{%
J.~Johnsen,\address{Department of Mathematical Sciences, Aalborg University,
Fredrik Bajers Vej 7G, DK-9220 Aalborg {\O}st, Denmark} 
\email{jjohnsen@math.aau.dk} 
S.~Munch~Hansen,
\address{Department of Mathematical Sciences, Aalborg University,
Fredrik Bajers Vej 7G, DK-9220 Aalborg {\O}st, Denmark} 
\email{sabrina@math.aau.dk}
 W.~Sickel
\address{Mathematisches Institut, Ernst-Abbe-Platz 2, D-07740 Jena, Germany} 
\email{Winfried.Sickel@uni-jena.de}
}
\thanks{\copyright\,2016 by the authors.
\\[3\baselineskip]{\tt Published under open access in Journal of function spaces, Article ID 964794, 2014.}}

\begin{document}
 \begin{abstract}
This article gives general results on invariance of anisotropic Lizorkin--Triebel spaces with  
mixed norms under coordinate transformations on Euclidean space, open sets 
and cylindrical domains.
 \end{abstract}
\enlargethispage{3\baselineskip} 
\maketitle
 
\section{Introduction}
This paper continues a study of anisotropic Lizorkin--Triebel spaces 
$F^{s,\vec a}_{\vec p,q}(\Rn)$ with mixed norms, which was begun in
\cite{JoSi07,JoSi08} and followed up in our joint work \cite{JoHaSi12}.

First Sobolev embeddings and completeness of the scale $F^{s,\vec a}_{\vec p,q}(\Rn)$ 
were established in \cite{JoSi07}, using the Nikol'ski\u\i--Plancherel--Polya inequality for
sequences of functions in the mixed-norm space $L_{\vec p}(\Rn)$,  
which was obtained straightforwardly in \cite{JoSi07}. Then a detailed 
trace theory for hyperplanes in $\Rn$ was worked out in \cite{JoSi08}, 
e.g.\ with the novelty that the well-known borderline $s=1/p$ has to
be shifted upwards in some cases, because of the mixed norms. 

Secondly, our joint paper \cite{JoHaSi12} presented some general characterisations of 
$F^{s,\vec a}_{\vec p,q}(\Rn)$, which may be specialised to kernels of
local means, in Triebel's sense \cite{T92}.  
One interest of this is that local means have recently been useful for obtaining
wavelet bases of Sobolev spaces and especially of their generalisations to the Besov and
Lizorkin--Triebel scales. Cf.\ works of
Vybiral~\cite[Th.~2.12]{Vyb06}, 
Triebel~\cite[Th.~1.20]{Tri08}, Hansen~\cite[Th.~4.3.1]{Han10}.

In the present paper, we treat the invariance of $F^{s,\vec a}_{\vec
p,q}$ under coordinate changes. During the discussions below, the results in \cite{JoHaSi12}
are crucial for the entire strategy. 

Indeed, we address the main technical challenge to obtain invariance
of $F^{s,\vec a}_{\vec p,q}(\Rn)$ under the map 
\begin{equation*}
  f\mapsto f\circ\sigma,  
\end{equation*}
when $\sigma$ is a bounded diffeomorphism on $\Rn$. (Cf.\ Theorem~\ref{coo1}, \ref{coo2-thm} below.)
Not surprisingly, this will require the condition on $\sigma$ that it only affects
blocks of variables $x_j$ in which the corresponding integral
exponents $p_j$ are equal, and similarly for the anisotropic weights
$a_j$. Moreover, when estimating the operator norm of $f\mapsto f\circ \sigma$, 
i.e.\ obtaining the inequality 
\begin{equation}
  \|\,f\circ\sigma\,|F^{s,\vec a}_{\vec p,q}(\Rn)\|\le c   \|\,f\,|F^{s,\vec a}_{\vec p,q}(\Rn)\|,
\end{equation}
the Fourier analytic
definition of the spaces seems difficult to manage directly, so as done by Triebel~\cite{T92}
we have chosen to characterise $F^{s,\vec a}_{\vec p,q}(\Rn)$ in
terms of local means, as developed in \cite{JoHaSi12}. 

However, the diffeomorphism invariance relies not just on the local means,
but first of all also on techniques underlying them. In particular, we use the
following inequality for the maximal function $\psi_j^*f(x)$  
of Peetre-Fefferman--Stein type, which was established
in \cite[Th.~2]{JoHaSi12} for mixed norms and with uniformity with
respect to a general parameter $\theta$: 
\begin{gather*}
\big\|\, \{2^{sj} 
\sup_{\theta \in \Theta} \psi_{\theta,j}^* f\}_{j=0}^\infty \, \big| L_{\vec{p}}(\ell_q)\big\|
\le  c \, \big\|\, \{2^{sj} \varphi^*_j f\}_{j=0}^\infty \, \big| L_{\vec{p}}(\ell_q)\big\| .
\end{gather*}
Hereby the `cut-off' functions $\psi_j$, $\varphi_j$ should fulfill a
set of Tauberian and moment conditions; 
cf.\ Theorem~\ref{maxmod} below for the full statement. 
In the isotropic case this inequality originated in a well-known
article of Rychkov~\cite{Ry}, which contains a serious flaw (as pointed out in \cite{Han10});
this and other inaccuracies were corrected in \cite{JoHaSi12}.

A second adaptation of Triebel's approach is caused by the anisotropy 
$\vec a$ we treat here. In fact, our proof only extends to
e.g. $s<0$ by means of the unconventional lift operator

\begin{equation}
  \Lambda_r=\operatorname{OP}(\lambda_r),\qquad
  \lambda_r(\xi)=\sum_{j=1}^n(1+\xi_j^2)^{\frac{r}{2a_j}}.
\end{equation}
Moreover, to cover all $\vec a=(a_1,\dots,a_n)$, especially to allow irrational ratios
${a_j}/{a_k}$, we found it useful to invoke the corresponding 
pseudo-differential operators $(1-\partial_j^2)^\mu=\operatorname{OP}((1+\xi_j^2)^{\mu})$ 
that for $\mu\in\R$ are shown here to be bounded 
$F^{s,\vec a}_{\vec p,q}(\Rn)\to F^{s-2a_j\mu,\vec a}_{\vec p,q}(\Rn)$ for all $s$.

Local versions of our result, in which $\sigma$ is only defined on subsets of $\Rn$, 
are also treated below. In short form we have e.g.\ the following result
(cf.\ Theorem~\ref{thm:localVersion} below):

\begin{theorem*}
Let $U,V\subset \Rn$ be open and let $\sigma: U\to V$ be a $C^\infty$-bijection on the form 
$\sigma(x)=(\sigma'(x_1,\dots,x_{n-1}),x_n)$.
When $f\in F^{s,\vec a}_{\vec p,q}(V)$ has compact support and all $p_j$ are equal for $j<n$, 
and similarly for the $a_j$, 
then $f \circ \sigma \in F^{s,\vec a}_{\vec p,q}(U)$ and 
\begin{equation}
  \|\, f\circ \sigma\, | F^{s,\vec a}_{\vec p,q}(U)\| 
  \leq c(\supp f, \sigma) \|\, f\, | F^{s,\vec a}_{\vec p,q}(V)\|.
\end{equation}
\end{theorem*}

This is useful for introduction of Lizorkin--Triebel spaces on 
cylindrical manifolds.
However, this subject is postponed to our forthcoming paper \cite{JoMHSi2}. 
(Already this part of the mixed-norm theory has seemingly not been elucidated before).
Moreover, in \cite{JoMHSi2} we also carry over trace results from \cite{JoSi08} to spaces over  
a smooth cylindrical domain in Euclidean space e.g.\ by analysing boundedness and ranges for traces 
on the flat and curved parts of its boundary.

To elucidate the importance of the results here and in \cite{JoMHSi2}, we recall that the
$F^{s,\vec a}_{\vec p,q}$ are relevant for parabolic differential equations
with initial and boundary value conditions: when solutions are
sought in a \emph{mixed-norm} Lebesgue space $L_{\vec p}$ 
(in order to allow different properties in the space and time directions), then
$F^{s,\vec a}_{\vec p,q}$-spaces are in general \emph{inevitable} for a correct
description of non-trivial data on the \emph{curved} boundary.

This conclusion was obtained in works of P.~Weidemaier \cite{Wei98,Wei02,Wei05}, 
who treated several special cases; one may also consult the introduction of
\cite{JoSi08} for details.

\subsection*{Contents}
Section~\ref{mixd} contains a review of our notation, and the definition of anisotropic 
Lizorkin--Triebel spaces with mixed norms is recalled, together with some needed properties, 
a discussion of different lift operators and a pointwise multiplier assertion.

In Section~\ref{localMeans} results from~\cite{JoHaSi12} on
characterisation of $F^{s,\vec a}_{\vec p,q}$-spaces 
by local means are recalled and used to prove an important lemma for compactly supported 
elements in $F^{s,\vec a}_{\vec p,q}$.
Sufficient conditions for $f \mapsto f \circ \sigma$
to leave the spaces $F^{s,\vec a}_{\vec p,q}(\Rn)$ invariant for all $s\in\R$
are deduced in Section~\ref{approach1}, when $\sigma$ is a bounded diffeomorphism.
Local versions for spaces on domains are derived in
Section~\ref{derived-sect} together with isotropic results. 

\section{Preliminaries}
\label{mixd}
\subsection{Notation}
The Schwartz space $\cs(\Rn)$ contains all rapidly decreasing $C^\infty$-functions.
It is equipped with the family of seminorms, using $D^\alpha :=
(-i\partial_{x_1})^{\alpha_1}\cdots(-i\partial_{x_n})^{\alpha_n}$ for each multi-index 
$\alpha=(\alpha_1,\ldots,\alpha_n)$ with
$\alpha_j\in\N_0:=\N\cup\{0\}$, and $\langle x\rangle^2 := 1+|x|^2$, 
\begin{equation}\label{eq:seminormOneParam}
p_M(\varphi) := \sup \big\{\, \langle x\rangle^M |D^\alpha
\varphi(x)|\, \big|\, x\in\Rn, |\alpha|\leq M\big\},\quad M\in\N_0; 
\end{equation}
or with 
\begin{equation}\label{eq:seminormsS_Nalpha}
q_{N,\alpha}(\psi) := \int_{\R^n} \langle x\rangle^N |D^\alpha\psi(x)|\, dx ,
\quad N\in\N_0,\enskip \alpha\in\N_0^n.
\end{equation}
The Fourier transformation $\cf g(\xi)=\widehat{g}(\xi) = \int_{\Rn} e^{-\mathrm{i} x\cdot \xi} g(x)\, dx$ for $g\in\cs(\Rn)$
extends by duality to the dual space $\cs'(\Rn)$ of temperate distributions.

Inequalities for vectors $\vec{p}=(p_1,\ldots,p_n)$ are understood componentwise; as are functions,
e.g.~$\vec{p}\,!=p_1!\cdots p_n!$. Moreover, $t_+:=\max(0,t)$ for $t\in\R$.

For $0<\vec p\leq \infty$ the space $L_{\vec{p}}(\Rn)$ consists of all
Lebesgue measurable functions such that  
\begin{alignat}{1}
\|\, u \, | L_{\vec{p}}(\Rn)\| := 
\bigg(\int_{-\infty}^\infty \bigg( \ldots 
\bigg( \int_{-\infty}^\infty |u(x_1,\ldots ,x_n)|^{p_1} dx_1
\bigg)^{p_2/p_1}  
\ldots  
\bigg)^{p_n/p_{n-1}} dx_n \bigg)^{1/p_n} <\infty,
\end{alignat}
with the modification of using the essential supremum over $x_j$ in case $p_j=\infty$.
Equipped with this quasi-norm, $L_{\vec{p}}(\Rn)$ is a 
quasi-Banach space (normed if $p_j\ge 1$ for all $j$).

Furthermore, for $0 < q \le \infty$ we shall use the notation
$\lpvec (\ell_q)(\Rn)$ for the space of all sequences 
$\{u_k\}_{k=0}^\infty$ of Lebesgue measurable functions $u_k: \Rn \to \C$ such that
\begin{equation}
\| \, \{u_k\}_{k=0}^\infty \, |\lpvec (\ell_q)(\Rn)\| :=
\Big\| \, \Big(\sum_{k=0}^\infty |u_k(\cdot)|^q
\Big)^{1/q}\, \Big|\lpvec (\Rn)\Big\| < \infty ,
\end{equation}
with supremum over $k$ in case $q=\infty$.
This quasi-norm is often abbreviated to $\| \, u_k \, |\lpvec (\ell_q)\| $;
and when $\vec{p}=(p,\ldots,p)$ we simplify $L_{\vec{p}}$ to $L_p$.
If $\max (p_1, \ldots, p_n, q)<\infty$  sequences of 
$C_0^\infty$-functions  are dense in $\lpvec (\ell_q)$. 

Generic constants will primarily be denoted by $c$ or $C$ and when 
relevant, their dependence on certain parameters will be explicitly stated.
$B(0,r)$ stands for the ball in $\Rn$ centered at $0$ with radius $r>0$, 
and $\overline{U}$ denotes the closure of a set $U\subset\Rn$.

\subsection{Anisotropic Lizorkin--Triebel Spaces with Mixed Norms}
The scales of mixed-norm Lizorkin--Triebel spaces refines the scales of 
mixed-norm Sobolev spaces, cf.~\cite[Prop.~2.10]{JoSi08}, hence
the history of these spaces goes far back in time; the reader is referred
to~\cite[Rem.~2.3]{JoHaSi12} and~\cite[Rem.~10]{JoSi07} for a brief historical overview,
which also list some of the ways to define Lizorkin--Triebel spaces.

Our exposition uses the Fourier-analytic definition, but first we recall the definition of
the anisotropic distance function $|\cdot|_{\vec a}$, where
$\vec{a}=(a_1,\ldots,a_n)\in [1,\infty[^n$,  
on $\Rn$ and some of its properties. Using the quasi-homogeneous dilation 
$t^{\vec a}x:=(t^{a_1}x_1,\dots,t^{a_n}x_n)$ for $t\ge0$, $|x|_{\vec a}$
is for $x\in\Rn\setminus\{0\}$ defined as the unique $t>0$ such that
$t^{-\vec a}x\in S^{n-1}$ ($|0|_{\vec a}:=0$), i.e.\ 
\begin{equation}
  \frac{x_1^2}{t^{2a_1}}+\dots+  \frac{x_n^2}{t^{2a_n}}=1.
\end{equation}
By the Implicit Function Theorem, $|\cdot|_{\vec a}$ is $C^\infty$ on $\Rn\setminus\{0\}$.
We also recall the quasi-homogeneity $|t^{\vec a}x|_{\vec a}=t|x|_{\vec a}$
together with (cf.~\cite[Sec.~3]{JoSi07})
\begin{align}
  |x+y|_{\vec a}&\le |x|_{\vec a}+|y|_{\vec a},
\label{atriangle-ineq}\\
  \max(|x_1|^{1/a_1},\dots,|x_n|^{1/a_n})
  &\le |x|_{\vec a}\le |x_1|^{1/a_1}+\dots+|x_n|^{1/a_n}.\label{ax-ineq}  
\end{align} 

The definition of $F^{s,\vec a}_{\vec p,q}(\Rn)$ uses a
Littlewood-Paley decomposition, i.e.~$1=\sum_{j=0}^\infty \Phi_j(\xi)$, 
which (for convenience) is based on a  fixed $\psi\in C_0^\infty$ such
that $0\leq\psi(\xi)\leq 1$ for all $\xi$, $\psi(\xi)=1$ if
$|\xi|_{\vec a}\leq 1$ and $\psi(\xi)=0$ if $|\xi|_{\vec a}\geq 3/2$;
setting $\Phi = \psi - \psi(2^{\vec a}\cdot)$, we define 
\begin{equation}\label{unity}
\Phi_0 (\xi) := \psi (\xi), \quad \Phi_j (\xi) := \Phi(2^{-j\vec a}\xi), \quad j=1,2,\ldots 
\end{equation}

\begin{defn} \label{F-defn}
The Lizorkin--Triebel space $F^{s,\vec a}_{\vec p,q}(\Rn)$ with $s\in \R$, $0<\vec{p} < \infty$ and 
$0 < q \le \infty$ consists of all $u\in\cs'(\R^n)$ such that
\begin{equation*}
\| \, u \, | F^{s,\vec a}_{\vec p,q}(\Rn)\| := 
\Big\|\, \Big(\sum_{j=0}^\infty 2^{jsq} \left|\cfi \left(\Phi_j (\xi) 
\cf u(\xi)\right) ( \cdot )\right|^q \Big)^{1/q} \, \Big| L_{\vec{p}}(\Rn)
\Big\| < \infty .
\end{equation*}
\end{defn}

The number $q$ is called the sum exponent and the entries in $\vec p$
are integral exponents, while $s$ is a smoothness index. Usually the statements are valid for the
full ranges $0<\vec p<\infty$, $0<q\le\infty$, so we refrain from repeating these. Instead we
focus on whether $s\in\R$ is allowed or not. In the isotropic case, i.e.~$\vec a = (1,\ldots,1)$,
the parameter $\vec a$ is omitted. 

We shall also consider the closely related Besov spaces, recalled using the abbreviation
\begin{equation}\label{eq:defuj}
  u_j (x) := \cfi \left(\Phi_j (\xi)  \cf u(\xi)\right) (x), \quad x\in \Rn , \enskip j\in \N_0.
\end{equation}

\begin{defn}
The Besov space $B^{s,\vec a}_{\vec p,q}(\R^n)$ consists of
all $u\in\cs'(\R^n)$ such that
\begin{equation*}
  \|\, u\, | B^{s,\vec a}_{\vec p,q}(\R^n)(\R^n) \| :=
  \Big( \sum_{j=0}^\infty 2^{jsq}  \|\, u_j\, | L_{\vec{p}}(\R^n) \|^q\Big)^{1/q} <\infty.
\end{equation*}
\end{defn}

In~\cite{JoSi07,JoSi08} many results on these classes are elaborated,
hence we just recall a few facts.
They are quasi-Banach spaces (Banach spaces if $\min(p_1,\ldots,p_n,q)\geq 1$) and the 
quasi-norm is subadditive, when raised to the power $d:=\min(1,p_1,\ldots,p_n,q)$,
\begin{equation}
\| \, u+v \, |F^{s,\vec a}_{\vec p,q}(\Rn)\|^d \le 
\| \, u \, |F^{s,\vec a}_{\vec p,q} (\Rn)\|^d  +  \| \, v \, |F^{s,\vec a}_{\vec p,q} (\Rn)\|^d,\quad 
u,v \in F^{s,\vec a}_{\vec p,q}(\R^n).
\end{equation}
Also the spaces do not depend on the chosen anisotropic decomposition of unity (up to equivalent
quasi-norms) and there are continuous embeddings
\begin{equation}
  \label{eq:SFembedding}
  \cs (\Rn) \hookrightarrow F^{s,\vec a}_{\vec p,q} (\Rn) \hookrightarrow \cs' (\Rn),
\end{equation}
where $\cs$ is dense in $F^{s,\vec a}_{\vec p,q}$ for $q<\infty$.

Since for $\lambda>0$, the space $F^{s,\vec a}_{\vec p,q} $ coincides
with $F^{\lambda s, \lambda \vec{a}}_{\vec{p},q}$,  
cf.~\cite[Lem.~3.24]{JoSi08}, most results obtained for the scales when $\vec a\geq 1$ can be 
extended to the case $0< \vec a < 1$ 
(for details we refer to~\cite[Rem.~2.6]{JoHaSi12}).

The subspace $\lloc(\Rn)\subset \cd'(\Rn)$ of locally integrable functions is equipped
with the Fr\'{e}chet space topology defined from the seminorms
$u\mapsto  \int_{|x|\leq j} |u(x)|\, dx$, $j\in\N$.
By $C_{\operatorname{b}}(\Rn)$ we denote the Banach space of bounded, continuous
functions, endowed with the sup-norm.

\begin{lem}\label{derivative}
Let $s\in\R$ and $\alpha\in\N_0^n$ be arbitrary.
 \begin{itemize}
  \item[{\rm (i)}]
The differential operator $D^\alpha$ 
is bounded $F^{s,\vec a}_{\vec p,q}(\Rn)\to F^{s-\vec a\cdot \alpha,
  \vec{a}}_{\vec{p},q} (\Rn)$.
\item[{\rm (ii)}]
 For $s> \sum_{\ell = 1}^n \big(\frac{a_\ell}{p_\ell} - a_\ell\big)_+$
there is an embedding $F^{s,\vec a}_{\vec p,q} (\Rn)\hookrightarrow \lloc(\Rn)$.
\item[{\rm (iii)}] The embedding $F^{s,\vec a}_{\vec p,q}\hookrightarrow
C_{\operatorname{b}}(\Rn)$ holds true for $s>\frac{a_1}{p_1}+\dots+\frac{a_n}{p_n}$. 
  \end{itemize}
\end{lem}

\begin{proof}
For part (i) the reader is referred to \cite[Lem.~3.22]{JoSi08}, where a proof using
standard techniques for $F^{s,\vec a}_{\vec p,q}$ is indicated (though the
reference should have been to Proposition 3.13 instead of 3.14 there).

Part (ii) is obtained from the Nikol'skij inequality,
cf.~\cite[Cor.~3.8]{JoSi07}, which allows a reduction to the
case in which $p_j\geq 1$ for $j=1,\ldots,n$, while $s>0$; then the
claim follows from the embedding $F^{0,\vec{a}}_{\vec{p},1}\,
\hookrightarrow \lloc$.
Part (iii) follows at once from \cite[(3.20)]{JoSi08}.
\end{proof}

A local maximisation over a ball can be estimated in $\lpvec$, at
least for functions in certain subspaces of $C_{\operatorname{b}}(\Rn)$;
cf.\ Lemma~\ref{derivative}(iii): 

\begin{lem}[\cite{JoHaSi12}]\label{lem:lpLTEstimate}
When $C>0$ and $s>\sum_{l=1}^n \frac{a_l}{\min(p_1,\ldots,p_l)}$, then
\begin{equation}
\label{eq:supBall}
\Big\|\, \sup_{|x-y| < C} |u(y)| \, \Big| L_{\vec{p}}(\R^n_x)\Big\|\le c \,
\|\, u \, | F^{s,\vec{a}}_{\vec{p},q}\|.
\end{equation}
\end{lem}

Next we extend a well-known embedding to the mixed-norm setting.
Let $C^\rho_*(\Rn)$ denote the H\"{o}lder 
class of order $\rho>0$, which
by definition consists of all $u\in C^k(\Rn)$ satisfying
\begin{equation}
  \|u\|_\rho := \sum_{|\alpha|\leq k} \sup_{x\in\R^n} |D^\alpha u(x)| + 
  \sum_{|\alpha|=k} \sup_{x-y\in\R^n\setminus\{0\}} |D^\alpha u(x)-D^\alpha u(y)|\, |x-y|^{k-\rho} < \infty,
\end{equation}
whereby $k$ is the integer satisfying $k<\rho\leq k+1$. 

\begin{lem}
\label{holderEmb}
For $\rho>0$ and $s\in\R$ with $s\leq\rho$ there is an embedding
$C^\rho_*(\Rn) \hookrightarrow B^{s,\vec{a}}_{\infty,\infty}(\Rn)$.
\end{lem}

\begin{proof} 
The claim follows by modifying \cite[Prop.~8.6.1]{Hor97} to the anisotropic case, i.e.
\begin{equation}\label{eq:proofHolderEmb}
  \|\, u\, |  B^{s,\vec{a}}_{\infty,\infty}\| = \sup_{j\in \N_0} 2^{s j} \sup_{x\in\R^n} |\cfi
  (\Phi_j \cf u)(x)|  \leq c_\rho \|u\|_\rho .
\end{equation}
The expressions in the Besov norm are for $j\geq 1$ estimated using that 
$\cfi\Phi$ has vanishing moments of arbitrary order,
\begin{equation}
  \cfi (\Phi_j \cf u)(x) 
  = \int \cfi\Phi(y) \Big( u(x-2^{-j\vec{a}}y) - 
  \sum_{|\alpha|\le k}\frac{\partial^\alpha u(x)}{\alpha!} (-2^{-j\vec{a}}y)^\alpha \Big)\, dy.
\end{equation}
Using a Taylor expansion of order $k-1$ with $k\in\N$ chosen such that 
$k < \rho \leq k+1$ (or directly if $k=0$), we get an estimate of the parenthesis by
\begin{equation}
  \begin{split}
 \Big| \sum_{|\alpha|=k} \frac{k}{\alpha!} \big(-&2^{-j\vec{a}}y\big)^{\alpha} 
     \int_0^1 (1-\theta)^{k-1} \left( \partial^\alpha u(x-2^{-j\vec{a}}\theta y) 
     - \partial^\alpha u(x)\right) d\theta\Big|\\
  &\leq \sum_{|\alpha|=k} \frac{k}{\alpha!} \left|2^{-j\vec{a}}y\right|^k 
  \|u\|_\rho \left| 2^{-j\vec{a}}y\right|^{\rho-k} \int_0^1 (1-\theta)^{k-1} d\theta
  \leq c'_\rho \left| 2^{-j\vec{a}}y\right|^\rho \|u\|_\rho.
\end{split}
\end{equation}
Now we obtain, since $\vec{a}\geq 1$,
\begin{equation}
  \sup_{x\in \R^n} |\cfi (\Phi_j \cf u)(x)| \leq c'_\rho\, 2^{-j\rho} \|u\|_\rho \int
  |\cfi\Phi(y)||y|^\rho\, dy    \leq c_\rho\, 2^{-j\rho} \|u\|_\rho.
\end{equation}
This bound can also be used for $j=0$, if $c_\rho$ is large enough, so~\eqref{eq:proofHolderEmb}
holds for $\rho\geq s$. 
\end{proof}

As a tool we also need to know the mapping properties of certain Fourier
multipliers $\lambda(D)u:= \cfi(\lambda(\xi)\hat u(\xi))$. For generality's
sake, we give

\begin{prop} \label{Fmult-prop}
  When $\lambda\in C^\infty(\Rn)$ for some $r\in\R$ has finite seminorms of the form
  \begin{equation}
    C_\alpha(\lambda):=\sup\big\{\,2^{-j(r-\vec a\cdot\alpha)} |D^\alpha\lambda(2^{j\vec a}\xi)|
    \bigm|
    j\in\N_0,\ \frac14\le|\xi|_{\vec a}\le 4\,\big\},
  \qquad \alpha\in\N_0^n,
  \end{equation}
then $\lambda(D)$ is continuous on $\cs'(\Rn)$ and bounded
$F^{s,\vec a}_{\vec p,q}(\Rn)\to F^{s-r,\vec a}_{\vec p,q}(\Rn)$ for
all $s\in\R$, with
operator norm $\|\lambda(D)\|\le c_{\vec p,q}\sum_{|\alpha|\le N_{\vec p,q}}C_\alpha(\lambda)$.
\end{prop}
\begin{proof}
  The quasi-homogeneity of $|\cdot|_{\vec a}$ yields that
  $|D^\alpha\lambda(\xi)|\le cC_\alpha(\lambda)(1+|\xi|_{\vec
    a})^{r-\vec a\cdot\alpha}$, hence every derivative is
  of polynomial growth, cf.\ \eqref{ax-ineq}, so $\lambda(D)$ is a
  well-defined continuous map on $\cs'$.
Boundedness follows as in the proof of \cite[Prop.~3.15]{JoSi08},
mutatis mutandis. In fact, only the last step there needs an
adaptation to the symbol $\lambda(\xi)$, but this is trivial because
finitely many of the constants $C_\alpha(\lambda)$ can enter the estimates.
\end{proof}

\subsection{Lift Operators} \label{lift-ssect}
The invariance under coordinate transformations will be established below using a somewhat
unconventional lift operator $\Lambda_r$, $r\in\R$,
\begin{equation}
  \Lambda_r u =\operatorname{OP}(\lambda_r(\xi))u= \cfi \big( \lambda_r(\xi) \widehat{u}(\xi) \big), 
  \qquad \lambda_r(\xi) = \sum_{k=1}^n (1+ \xi_k^2)^{r/(2a_k)}.  
\end{equation}
To apply Proposition~\ref{Fmult-prop}, we derive an
estimate uniformly in $j\in\N_0$ and over the set $\frac14\le|\xi|_{\vec a}\le 4$: 
while the mixed derivatives vanish, the explicit higher order chain rule in Appendix~A yields
\begin{equation}
   |D^{\alpha_l}_{\xi_l}(2^{-jr}\lambda_r(2^{j\vec a}\xi))| \le
  \sum_{k=1}^{\alpha_l} c_k (2^{-2ja_l}+ \xi_l^2)^{\frac{r}{2a_l}-k}
  2^{j(\alpha_la_l-2k a_l)}
  \sum_{\substack{k=n_1+n_2\\ \alpha_l=n_1+2n_2}}(2(2^{ja_l}\xi_l))^{n_1} 2^{n_2}
<\infty.
\label{Lambda-est}
\end{equation}
Indeed, the precise summation range gives $\alpha_l=n_1+2(k-n_1)$,
so the harmless power $2^{n_1+n_2}$ results. (Note that this means that
$|D^\alpha\lambda_r(2^{j\vec a}\xi)|\le C_\alpha 2^{j(r-\vec a\cdot\alpha)}$.)

Now $\lambda_r(\xi)$ has no zeros, and for $\lambda_r(\xi)^{-1}$ 
it is analogous to obtain such estimates uniformly with respect to $j$ of
$D^\alpha(2^{jr}\lambda_r(2^{j\vec{a}}\xi)^{-1})$, using Appendix~A and the above.
So Proposition~\ref{Fmult-prop} gives both that
$\Lambda_r$ is a homeomorphism on $\cs'$
(although $\Lambda_r^{-1}\ne\Lambda_{-r}$) and the proof of

\begin{lem}\label{lift}
The map $\Lambda_r$ is a linear homeomorphism 
$F^{s,\vec a}_{\vec p,q}(\Rn) \to F^{s-r ,\vec{a}}_{\vec{p},q} (\Rn)$ for $s\in\R$. 
\end{lem}

In a similar way one also finds the next auxiliary result.

\begin{lem}\label{lem:lift}
For any $\mu\in\R$, $k\in\{1,\ldots,n\}$ the map 
$(1-\partial_{x_k}^2)^\mu u=\operatorname{OP}((1+\xi_k^2)^{\mu})u$ is a linear homeomorphism 
$F^{s,\vec a}_{\vec p,q}(\Rn)\to F^{s-2\mu a_k,\vec a}_{\vec p,q}(\Rn)$ for all $s\in\R$.
\end{lem}

A standard choice of an anisotropic lift operator is obtained 
by associating each $\xi\in\Rn$ with $(1,\xi)\in\R^{1+n}$, which is given the weights
$(1,\vec a)$, and by setting
\begin{equation}
  \langle\xi\rangle_{\vec a}=|(1,\xi)|_{(1,\vec a)}.
\end{equation}
This is in $C^\infty$, as $|\cdot|_{(1,\vec a)}$ is so outside the origin. 
(Note the analogy to $\langle\xi\rangle=\sqrt{1+|\xi|^2})$.
Moreover, $\partial^\alpha\langle\xi\rangle_{\vec a}^t$ is for each $t\in\R$ estimated by powers of
$|\xi|$, cf.~\cite[Lem.~1.4]{Y1}. 
Therefore there is a linear homeomorphism $\Xi_{\vec a}^t\colon \cs'\to\cs'$ given by
\begin{equation}
  \Xi_{\vec a}^t u := \operatorname{OP}(\langle\xi\rangle_{\vec a}^t)u=
   \cfi \left( \langle\xi\rangle_{\vec a}^t\, \widehat{u}(\xi) \right),
  \qquad t\in\R.
\end{equation}
In our mixed-norm set-up it is a small exercise to show that it restricts to a homeomorphism
\begin{equation}
  \Xi_{\vec a}^t\colon F^{s,\vec a}_{\vec p,q}(\Rn)\to F^{s-t,\vec a}_{\vec p,q}(\Rn)
  \quad\text{for all $s\in\R$}.
\label{Xi-id}
\end{equation}
Indeed, invoking Proposition~\ref{Fmult-prop}, the task is as in
\eqref{Lambda-est} to show a uniform bound, and using the
elementary properties of $\langle\xi\rangle_{\vec a}$ (cf.\ \cite[Lem.~1.4]{Y1}) one
finds for $t-\vec a\cdot\alpha\geq 0$,
\begin{equation}
  \big|D^\alpha(2^{-jt}\langle 2^{j\vec a}\xi\rangle_{\vec a}^t)\big|=
  2^{j(\vec a\cdot\alpha-t)}\Big| {D^\alpha_\eta\langle\eta\rangle_{\vec
      a}^t}{\bigm|_{\eta=2^{j\vec a}\xi}}\Big| 
\le
  c2^{j(\vec a\cdot\alpha-t)}\langle 2^{j\vec a}\xi\rangle_{\vec a}^{t-\vec a\cdot\alpha}
\le c\langle\xi\rangle^{t-\vec a\cdot\alpha}.
\label{Xi-est}
\end{equation}
When $t-\vec a\cdot\alpha\leq 0$, then  $|\xi|_{\vec a}^{t-\vec a\cdot\alpha}$ is the outcome on
the right-hand side. But the uniformity results in both cases, 
since the estimates pertain to $\frac14\le|\xi|_{\vec a}\le 4$.

We digress to recall that the classical fractional Sobolev space
$H^{s,\vec a}_{\vec p}(\Rn)$, for $s\in\R$ 
and $1<\vec p<\infty$, consists of the $u\in\cs'$ for which $\Xi_{\vec a}^su\in L_{\vec p}(\Rn)$; with 
$\|\,u\,| H^{s,\vec a}_{\vec p}\|:=\|\,\Xi_{\vec a}^su\,| L_{\vec p}\|$. 
If $m_k:=s/a_k\in \N_0$ for all $k$, then $H^{s,\vec a}_{\vec p}$  coincides 
(as shown by Lizorkin~\cite{Li})
with the space $W^{(m_1,\ldots,m_n)}_{\vec p}(\Rn)$ of $u\in\lpvec$ having 
$\partial_{x_k}^{m_k}u$ in $\lpvec$ for all $k$.

This characterisation is valid for $F^{s,\vec a}_{\vec p,2}$ with $1<\vec p<\infty$
in view of the identification
\begin{equation}
  u\in H^{s,\vec a}_{\vec p}(\Rn)  \iff   u\in F^{s,\vec{a}}_{\vec{p},2} (\Rn),
\label{Hsp-id}
\end{equation}
which by use of $\Xi^s$ reduces to the case $L_{\vec p}=F^{0,\vec a}_{\vec p,2}$. 
The latter is a Littlewood-Paley inequality that may be proved
with general methods of harmonic analysis; cf.~\cite[Rem.~3.16]{JoSi08}.

A general reference on mixed-norm Sobolev spaces is 
the classical book of Besov, Ilin and Nikolskii~\cite{BIN79,BIN96}.
Schmeisser and Triebel~\cite{ScTr87} treated $F^{s,\vec a}_{\vec p,q}$ for $n=2$.

\begin{rem}
Traces on hyperplanes were considered for $H^{s,\vec a}_{\vec p}(\Rn)$ by Lizorkin~\cite{Li}
and for $W^{\vec m}_{\vec p}(\Rn)$ by Bugrov~\cite{Bug71}, who raised the problem of traces at
$\{x_j=0\}$ for $j<n$. This was solved by Berkolaiko, who treated traces in
the $F^{s,\vec a}_{\vec p,q}(\Rn)$-scales for $1< \vec p <\infty$ in
e.g.~\cite{Ber85}.
The range $0<\vec p<\infty$ was covered on $\Rn$ for $j=1$ and $j=n$ in \cite{JoSi08}, and
in our forthcoming paper \cite{JoMHSi2} we carry over the trace results to $F^{s,\vec a}_{\vec
  p,q}$-spaces  
over a smooth cylindrical domain $\Omega\times]0,T[$.
\end{rem}

\begin{rem} \label{JS2-rem}
We take the opportunity to correct a minor inaccuracy in~\cite{JoSi08}, where 
a lift operator (also) called $\Lambda_r$  unfortunately was defined to have symbol
$(1+|\xi|_{\vec a}^2)^{r/2}$. However, it is not in $C^\infty(\Rn)$ for $\vec a\ne(1,\dots,1)$;
this can be seen  from the example for $n=2$ with $\vec a=(2,1)$ where~\cite[Ex.~1.1]{Y1} gives
the explicit formula
\begin{equation}
  |\xi|_{\vec a}= 2^{-1/2}\big(\xi_2^2+(\xi_2^4+4\xi_1^2)^{1/2}\big)^{1/2}.
\end{equation}
Here an easy calculation shows that $D_{\xi_1}|\xi|_{\vec a}^2$ is discontinuous along the line
$(\xi_1,0)$, which is inherited by the symbol e.g.\ for $r=2$. The
resulting operator is therefore not defined on all of $\cs'$. However, this is
straightforward to avoid by replacing the lift operator
in~\cite{JoSi08} by the better choice $\Xi^r$  given in \eqref{Xi-id}.
This gives the space $H^{s,\vec a}_{\vec p}(\Rn)$ in \eqref{Hsp-id}.
\end{rem}

\subsection{Paramultiplication}
This section contains a pointwise multiplier assertion for the $F^{s,\vec a}_{\vec p,q}$-scales.
We consider the densely defined product on $\cs'\times\cs'$, introduced in 
\cite[Def.~3.1]{JJ94mlt} and in an isotropic set-up in \cite[Ch.~4]{RS},
\begin{equation}\label{eq:extendedProduct}
u \cdot v := \lim_{j \to \infty} \cfi \left(\psi (2^{-j\vec{a}}\xi)  \cf u (\xi)\right)
\cdot \cfi \left(\psi (2^{-j\vec{a}}\xi)  \cf v (\xi)\right),
\end{equation}
which is considered for those pairs $(u,v)$ in $\cs'\times\cs'$ for which 
the limit on the right-hand side exists in $\cd'$ and is independent of $\psi$.
Here $\psi\in C_0^\infty$ is the function used in the construction of the
Littlewood-Paley decomposition (in principle the independence should be verified for all
$\psi\in C_0^\infty$ equalling 1 near the origin; but this is not a problem here).

To illustrate how this product extends the usual one, and to prepare for an application
below, the following is recalled:

\begin{lem}[\cite{JJ94mlt}]\label{lem:paramultiplication} When $f\in C^\infty(\Rn)$ has derivatives of any order
of polynomial growth, and when $g\in\cs'(\Rn)$ is arbitrary, then the limit in 
\eqref{eq:extendedProduct} exists and equals the usual product $f\cdot g$, as defined on
$C^\infty \times \cd'$.
\end{lem}

Using this extended product, we introduce the usual space of multipliers
\begin{equation}
  M(F^{s,\vec a}_{\vec p,q}):= \big\{u \in \cs'\, \big|\,  
  u\cdot v \in F^{s,\vec a}_{\vec p,q} \mbox{ for all } v \in F^{s,\vec a}_{\vec p,q}\big\}
\end{equation}
equipped with the induced operator quasi-norm
\begin{equation}
\| \, u \, |M(F^{s,\vec a}_{\vec p,q})\|:= \sup \big\{\| \, u \cdot v\, |F^{s,\vec a}_{\vec p,q} \| \, \big|\,
\| \, v \, |F^{s,\vec a}_{\vec p,q} \| \le 1 \big\} .
\end{equation}

As Lemma~\ref{holderEmb} at once yields
$C^\infty_{L_\infty} \subset \bigcap\limits_{s>0} B^{s,\vec{a}}_{\infty,\infty}$ 
(a well-known result in the isotropic case)
for 
$C^\infty_{L_\infty} := \{\, g\in C^\infty\, |\,\forall\alpha\colon D^\alpha g\in L_\infty\}$,
the next result is in particular valid for  
$u\in C^\infty_{L_\infty}$:

\begin{lem}\label{mult}
Let $s \in \R$ and take $s_1 > s$ such that also
\begin{equation}
  \label{eq:assumptionMult}
s_1 > \sum_{\ell=1}^n \Big(\frac{a_\ell}{\min(1,q,p_1, \ldots , p_\ell)} -a_\ell \Big) -s.
\end{equation}
Then each $u\in B^{s_1, \vec{a}}_{\infty,\infty}$ defines a multiplier of $F^{s,\vec a}_{\vec p,q}$ and
\begin{equation}
\| \, u \, | M(F^{s,\vec a}_{\vec p,q})\| \le c \, 
\| \, u \, | B^{s_1, \vec{a}}_{\infty,\infty}\|.
\end{equation}
\end{lem}

\begin{proof}
The proof will be brief as it is based on standard arguments from paramultiplication,
cf.~\cite{JJ94mlt} and \cite[Ch.~4]{RS} for details.
In particular we shall use the decomposition 
\begin{equation}
  u\cdot v = \Pi_1(u,v) + \Pi_2(u,v) + \Pi_3(u,v).
\end{equation}
The exact form of this can also be recalled from the below formulae. In terms of the
Littlewood-Paley partition $1=\sum_{j=0}^\infty \Phi_j(\xi)$ from
Definition~\ref{F-defn},
we set $\Psi_j = \Phi_0 + \dots +\Phi_j$ for $j\geq 1$ and $\Psi_0=\Phi_0$. 
These are used in Fourier 
multipliers, now written with upper indices as $u^j =\cf^{-1}(\Psi_j \widehat{u})$.

Note first that $s_1>0$, whence $ B^{s_1,\vec{a}}_{\infty,\infty} \hookrightarrow L_\infty$,
which is useful since the dyadic corona criterion for $F^{s,\vec a}_{\vec p,q}$,
cf.~\cite[Lem.~3.20]{JoSi08}, implies the well-known simple estimate 
\begin{equation}
\| \, \Pi_1 (u,v)\, | F^{s,\vec a}_{\vec p,q} \| \le c \, \| \, u \, | L_\infty \|
\,  \| \, v \, | F^{s,\vec a}_{\vec p,q} \|.
\end{equation}
Furthermore, since 
\begin{equation}
s_2:= s_1 + s > \sum_{\ell=1}^n \frac{a_\ell}{\min(1,q,p_1, \ldots\, , p_\ell)} - |\vec{a}|,
\end{equation}
using the dyadic ball criterion for $F^{s,\vec a}_{\vec p,q}$, cf.~\cite[Lem.~3.19]{JoSi08},
we find that
\begin{equation}
  \begin{split}
\| \, \Pi_2 (u,v)\, | F^{s_2,\vec a}_{\vec p,q}\|
& \le  c \, \| \, 2^{js_2} u_j v_j \, 
| L_{\vec{p}}(\ell_q) \|
\\
& \le   c \, \sup_{k\in\N_0} 2^{ks_1} \|\, u_k\, | L_\infty\|\, 
\| 2^{js}|v_j| \, | L_{\vec{p}}(\ell_q) \|
\\
& \le   c \, \| \, u\, | B^{s_1, \vec{a}}_{\infty,\infty}\|
\,  \| \, v \, | F^{s, \vec{a}}_{\vec{p}, q} \|.
  \end{split}  
\end{equation}
To estimate $\Pi_3(u,v)$ we first consider the case $s>0$ and pick $t\in\, ]s,s_1[\,$. The
dyadic corona criterion together with the formula $v^j = v_0+\dots +
v_j$ and a summation lemma, which exploits that $t-s_1<0$ (cf.~\cite[Lem.~3.8]{Y1}), give
\begin{gather}
\begin{split}
\label{eq:pi3estimate}
\| \, \Pi_3(u,v)\, |F^{t,\vec{a}}_{\vec{p},q}\|
&\leq c \sup_{k\in\N_0} 2^{ks_1} \|\, u_k\, |L_\infty\|\, \|2^{(t-s_1)j} v^{j-2}\, | L_{\vec{p}}(\ell_q)\|\\
&\leq c\, \|\, u\, | B^{s_1,\vec{a}}_{\infty,\infty}\|\, \Big\|\, 2^{(t-s_1)j} \sum_{k=0}^j |v_k|\, \Big| L_{\vec{p}}(\ell_q)\Big\|\\
&\leq    c\, \|\, u\, | B^{s_1,\vec{a}}_{\infty,\infty}\|\, \|\, v\, | F^{t-s_1,\vec{a}}_{\vec{p},q}\|.
\end{split}
\end{gather}
Since $t-s_1<0<s$ implies $F^{s,\vec a}_{\vec p,q}\hookrightarrow F^{t-s_1,\vec{a}}_{\vec{p},q}$, and also
$F^{t,\vec{a}}_{\vec{p},q}\hookrightarrow F^{s,\vec a}_{\vec p,q}$ holds, the above yields
\begin{alignat}{1}
  \label{eq:pi3estimate2}
  \|\, \Pi_3(u,v)\, | F^{s,\vec a}_{\vec p,q}\| \leq c\, \|\, u\, | B^{s_1,\vec{a}}_{\infty,\infty}\|\, \|\, v\, | F^{s,\vec a}_{\vec p,q}\|.
\end{alignat}
For $s\leq 0$ the procedure is analogous, except that \eqref{eq:pi3estimate} is derived for
$t\in\, ]0,s_1+s[\,$, which is non-empty by assumption \eqref{eq:assumptionMult} on $s$; then
standard embeddings again give \eqref{eq:pi3estimate2}.

In closing, we remark that as required the product $u\cdot v$ is independent of the test function 
$\psi$ appearing in the definition. Indeed for $q<\infty$ this follows
from Lemma~\ref{lem:paramultiplication}, which gives the  
coincidence between this product on $\cs'\times \cs$ and the usual one,
hence by density of $\cs$, cf.~\eqref{eq:SFembedding}, and the above
estimates, the map $v\mapsto u\cdot v$  
extends uniquely by continuity to all $g\in F^{s,\vec a}_{\vec p,q}$. For $q=\infty$ the embedding 
$F^{s,\vec{a}}_{\vec{p},\infty}\hookrightarrow F^{s-\varepsilon,\vec{a}}_{\vec{p},1}$ 
for $\varepsilon>0$ yields the independence using the previous case.
\end{proof}

\section{Characterisation by Local Means}
\label{localMeans}

Characterisation of Lizorkin--Triebel spaces $F^s_{p,q}$ by local means is due
to Triebel,~\cite[2.4.6]{T92}, and it was from the outset an important tool in proving
invariance of the scale under diffeomorphisms. An extensive treatment of characterisations of
mixed-norm spaces  $F^{s,\vec a}_{\vec p,q}$ in terms of quasi-norms based on
convolutions, in particular the case of local means, was given
in~\cite{JoHaSi12}, which to a large extent is based on extensions to
mixed norms of inequalities in \cite{Ry}. For the reader's convenience
we recall the needed results. 

Throughout this section we consider a fixed anisotropy $\vec a\geq 1$
with $\underline{a} := \min (a_1, \ldots ,a_n)$ and functions $\psi_0, \psi \in \cs (\Rn)$ that 
fulfil Tauberian conditions in terms of some $\varepsilon >0$ and/or a moment 
condition of order $M_\psi\ge -1$ ($M_\psi = -1$ means that the condition is void),
\begin{align}
\label{con1}
|\cf \psi_0 (\xi) | & >  0 \quad \mbox{on}\quad \big\{\,\xi \, \big|\,  |\xi|_{\vec{a}}< 2 \varepsilon\big\},
\\
\label{con2}
|\cf \psi (\xi) | & > 0 \quad \mbox{on}\quad \{\,\xi \, |\,  
\varepsilon/2<|\xi|_{\vec{a}}< 2 \varepsilon\},\\
\label{moment}
D^\alpha (\cf \psi) (0)& =  0 \quad \mbox{for} \quad |\alpha|\le M_\psi.
\end{align}
Note by~\eqref{ax-ineq} that in case \eqref{con1} is fulfilled for the Euclidean distance, it holds
true also in the anisotropic case, perhaps with a different $\varepsilon$.

We henceforth change notation, from~\eqref{eq:defuj}, to 
\begin{equation}\label{eq:defOfSubscript_j}
  \varphi_j(x) = 2^{j|\vec a|}\varphi(2^{j\vec a}x),\quad \varphi\in\cs,\enskip j\in\N,
\end{equation} 
which gives rise to the sequence $(\psi_j)_{j\in\N_0}$.
The non-linear Peetre-Fefferman-Stein maximal operators induced by $(\psi_j)_{j\in\N_0}$ are for an
arbitrary vector $\vec{r} = (r_1, \ldots  , r_n)>0$ and any $f \in \cs'(\Rn)$ given by (dependence on
$\vec a$ and $\vec r$ is omitted) 
\begin{alignat}{1}
\label{eq:maximalPFStype}
 \psi^*_j f(x)  =  \sup_{y \in \Rn}
 \frac{|\psi_j * f (y)|}{\prod\limits_{\ell=1}^n
(1+ 2^{ja_\ell}|x_\ell - y_\ell |)^{r_\ell}} , \quad x\in \Rn,\enskip j \in\N_0.
\end{alignat}
Later we shall also refer to the trivial estimate
\begin{equation}
  \label{*max-ineq}
  |\psi_j*f(x)|\le \psi_j^* f(x).
\end{equation}

Finally for an index set $\Theta$, we consider $\psi_{\theta,0},\psi_\theta\in\cs(\Rn)$, $\theta\in\Theta$, where the $\psi_\theta$ satisfy~\eqref{moment} for some $M_{\psi_\theta}$ independent of $\theta\in \Theta$, and also $\varphi_0,\varphi\in\cs(\Rn)$ that fulfil \eqref{con1}--\eqref{con2} in terms of an $\varepsilon'>0$.
Setting $\psi_{\theta,j}(x)=2^{j|\vec a|}\psi_\theta(2^{j\vec a}x)$ for $j\in\N$, we can state the first result relating different quasi-norms.

\begin{thm}[\cite{JoHaSi12}]\label{maxmod}
Let $0< \vec{p}< \infty$, $0<q\le \infty$ and $-\infty < s < (M_{\psi_\theta}+1) \underline{a}$. 
For a given $\vec{r}$ in \eqref{eq:maximalPFStype}
and an integer $M\geq -1$ chosen so large that $(M+1)\underline{a}-2\vec{a}\cdot\vec{r} +s> 0$, we assume that
\begin{alignat*}{2}
A  & :=  \sup_{\theta \in \Theta}\,  \max  \| \, D^\alpha \cf \psi_\theta\, | L_\infty\|&\, <\infty ,
\\
B & : =  \sup_{\theta \in \Theta}\,  \max \| \, (1+ |\xi|)^{M+1}\, D^\gamma \cf \psi_\theta (\xi)\, | L_1\|&\, <\infty ,
\\
C & :=  \sup_{\theta \in \Theta}\,  \max \| \, D^\alpha \cf \psi_{\theta,0}\, | L_\infty\|&\, <\infty ,
\\
D  & : =  \sup_{\theta \in \Theta}\,  \max \| \, (1+ |\xi|)^{M+1}\, D^\gamma \cf \psi_{\theta,0} (\xi)\, | L_1\|&\, <\infty ,
\end{alignat*}
where the maxima are over $\alpha$ such that $|\alpha|\leq M_{\psi_\theta}+1$ or $\alpha\leq \lceil\vec r + 2\rceil$, 
respectively over $\gamma$ with $\gamma_j\leq r_j+2$.
Then there exists a constant $c>0$ such that for $f \in \cs' (\Rn)$,
\begin{gather}
\begin{split}
\label{eq-8}
\big\|\, \{2^{sj} 
\sup_{\theta \in \Theta} \psi_{\theta,j}^* f\}_{j=0}^\infty \, \big| L_{\vec{p}}(\ell_q)\big\|
\le  c (A +B + C +D)\, \big\|\, \{2^{sj} \varphi^*_j f\}_{j=0}^\infty \, \big| L_{\vec{p}}(\ell_q)\big\| .
\end{split}
\end{gather}
\end{thm}

It is also possible to estimate the maximal function in terms of the
convolution appearing in its numerator:

\begin{thm}[\cite{JoHaSi12}]\label{inverse}
Let $\psi_0,\psi\in\cs(\Rn)$ satisfy the Tauberian conditions \eqref{con1}--\eqref{con2}.
When $0< \vec{p}<\infty$, $0<q\le \infty$, $-\infty < s < \infty$ and 
\begin{equation}
  \frac{1}{r_l} < \min (q,p_1, \ldots, p_n) , \quad l =1,\ldots , n 
\end{equation}
there exists a constant $c>0$ such that for $f\in\cs'(\Rn)$,
\begin{equation}\label{eq-inverse}
\big\|\, \{2^{sj} \psi^*_j f\}_{j=0}^\infty \, \big| L_{\vec{p}}(\ell_q)\big\|
\le c\,                    
\big\|\, \{2^{sj} \psi_j * f\}_{j=0}^\infty \, \big| L_{\vec{p}}(\ell_q)\big\|.
\end{equation}
\end{thm}

As a consequence of Theorems~\ref{maxmod},~\ref{inverse} (the first applied for a trivial index set like $\Theta=\{1\}$),
we obtain the characterisation of $F^{s,\vec a}_{\vec p,q}$-spaces by local means:

\begin{thm}[\cite{JoHaSi12}]\label{local}
Let $k_0,k^0 \in \cs$ such that $\int k_0(x)\, dx \neq 0 \neq \int k^0(x)\, dx$ and
set $k (x)= \Delta^N k^0(x)$ for some $N\in\N$.
When $0 < \vec{p} <\infty$, $0< q \le \infty$, and $-\infty < s < 2N\underline{a}$,
then a distribution $f \in \cs'(\Rn)$ belongs to $F^{s,\vec{a}}_{\vec{p},q}(\Rn)$ 
if and only if (cf.~\eqref{eq:defOfSubscript_j} for the $k_j$)
\begin{equation}\label{eq:charOfF}
\|\, f \, | F^{s,\vec{a}}_{\vec{p},q}\|^* :=
\| \, k_0 * f\, |L_{\vec{p}}\| + \| \{2^{sj} k_j * f\}_{j=1}^\infty \, | L_{\vec{p}}(\ell_q)\|<\infty .
\end{equation}
Furthermore, $\|\, f \, | F^{s,\vec{a}}_{\vec{p},q}\|^*$ is an equivalent
quasi-norm on $F^{s,\vec{a}}_{\vec{p},q}(\Rn)$.
\end{thm}

Application of Theorem~\ref{local} yields a useful result regarding Lizorkin--Triebel spaces
on open subsets, when these are defined by restriction, i.e.

\begin{defn}\label{def:Fsubspace}
Let $U \subset\R^n$ be open. The space  $F^{s,\vec a}_{\vec p,q}(U)$ is defined as the set of all
$u \in \cd'(U )$ such that there exists a distribution $f \in F^{s,\vec a}_{\vec p,q}({\R}^{n})$
satisfying
\begin{equation}\label{eq:distributionRestriction}
 f(\varphi) = u (\varphi) \quad \mbox{for all}\quad \varphi \in 
C_0^\infty (U) .
\end{equation}
We equip $F^{s,\vec a}_{\vec p,q}(U)$ with the quotient quasi-norm 
$\|\, u\, | F^{s,\vec a}_{\vec p,q}(U)\| = \inf_{r_U f=u} \|\,f\, |F^{s,\vec a}_{\vec p,q}(\Rn)\|$;
it is normed if $\vec p,q\geq 1$.
\end{defn} 

In \eqref{eq:distributionRestriction} it is tacitly understood that on the left-hand side $\varphi$
is extended by 0 outside $U$. For this we henceforth use the operator notation $e_U \varphi$.
Likewise $r_U$ denotes restriction to $U$, whereby $u=r_U f$ in \eqref{eq:distributionRestriction}.

The Besov spaces $B^{s,\vec a}_{\vec p,q}(U)$ on $U$ can be defined analogously.
The quotient norms have the well-known advantage that
embeddings and completeness can be transferred directly from the
spaces on $\Rn$. However, the spaces are probably of little interest,
if $\partial U$ does not satisfy some regularity conditions, because
we then expect (as in the isotropic case) that they do not coincide with those defined intrinsically.

\begin{lem}\label{lem:equalityInfnorm}
Let $U\subset \Rn$ be open and $r>0$. When $F^{s,\vec a}_{\vec p,q}(U)$ has the
infimum quasi-norm derived from the local means in Theorem~\ref{local}
fulfilling $\supp k_0,\supp k \subset B(0,r)$, and 
\begin{equation}
   \dist(\supp f,\Rn\setminus U)>2r
\end{equation}
holds for some $f\in F^{s,\vec a}_{\vec p,q}(U)$ with compact support, then
\begin{equation}\label{eq:equalityInfnorm}
  \|\, f\, | F^{s,\vec a}_{\vec p,q}(U)\| = \|\, e_U f\, | F^{s,\vec a}_{\vec p,q}(\Rn)\|.
\end{equation}
In other words, the infimum is attained at $e_U f$ for such $f$.
\end{lem}

\begin{proof}
For any other extension $\tilde{f}\in\cs'(\Rn)$ the difference
$g=\tilde f-e_Uf$ is non-zero in $\cs'(\Rn)$ and $\supp e_U f \cap \supp g = \emptyset$.
So by the properties of $r$,
\begin{equation}
  \supp (k_j * e_U f) \cap \supp (k_j * g) = \emptyset, \quad j\in\N_0.
\end{equation}
Since $g\neq 0$ there is some $j$ such that  $\supp(k_j*g)\ne\emptyset$, hence $k_j * g(x)\neq 0$ 
on an open set disjoint from $\supp (k_j *e_U f)$. This term therefore effectively contributes
to the $L_{\vec p}$-norm in~\eqref{eq:charOfF} and thus 
$\|\, \tilde f \, | F^{s,\vec a}_{\vec p,q}\|=\|\, e_U f + g\, | F^{s,\vec a}_{\vec p,q}\| > \|\, e_U f\, | F^{s,\vec a}_{\vec p,q}\|$,
which shows \eqref{eq:equalityInfnorm}.
\end{proof}

\section{Invariance under Diffeomorphisms}
\label{approach1}
The aim of this section is to show that $F^{s,\vec a}_{\vec p,q}(\Rn)$ is invariant
under suitable diffeomorphisms $\sigma:\Rn\to\Rn$ and from this deduce similar 
results in a variety of set-ups.

\subsection{Bounded Diffeomorphisms}
  \label{bDiff-ssect}
A one-to-one mapping $y=\sigma (x)$
of $ \Rn$ onto $\Rn $ is here called a diffeomorphism if the components
$\sigma_j:  \Rn \to \R$ have  classical derivatives $D^\alpha \sigma_j$
for all $\alpha\in \N^n$. We set $\tau(y)=\sigma^{-1}(y)$.

For convenience $\sigma$ is called a \emph{bounded diffeomorphism} when
$\sigma$ and $\tau$ furthermore satisfy
\begin{gather}
\label{det0}
C_{\alpha,\sigma} := \max_{j\in\{1,\ldots,n\}} \| \, D^\alpha \sigma_j\,| L_\infty \| <\infty ,
\\
C_{\alpha,\tau} := \max_{j\in\{1,\ldots,n\}} \| \, D^\alpha \tau_j\,| L_\infty \| <\infty.
\label{det0'}
\end{gather}
In this case there are obviously  positive constants 
(when $J\sigma$ denotes the Jacobian matrix)
\begin{equation}
\label{det}
c_\sigma := \inf_{x\in \Rn} |\det J \sigma (x)| >0 , \qquad
c_\tau := \inf_{y\in \Rn} |\det J \tau (y)| >0 .
\end{equation}
E.g., by the Leibniz formula for determinants, $c_\sigma\ge 1/(n!\prod_{|\alpha|=1}C_{\alpha,\sigma})>0$.

Conversely, whenever a $C^\infty$-map $\sigma\colon\Rn\to\Rn$ fulfils \eqref{det0} and $c_\sigma>0$, 
then $\tau$ is $C^\infty$
(as $J\tau(y) = \frac{1}{\det J\sigma(\tau(y))} \operatorname{Adj}J\sigma(\tau(y))$,
if $\operatorname{Adj}$ denotes the adjugate, each $\partial_j\tau_k$ is in $C^m$ if $\tau$ is so)
and using e.g.\ Appendix~\ref{app:higherOrderCR} it is seen by induction over $|\alpha|$ that it
fulfils \eqref{det0'}.  
Hence $\sigma$ is a bounded diffeomorphism.

Recall that for a bounded diffeomorphism $\sigma$ and a temperate distribution $f$, the composition
$f\circ \sigma $ denotes the temperate distribution given by
\begin{equation}\label{eq:defCompositionSprime}
\langle f\circ\sigma,\psi\rangle = \langle f,\psi \circ \tau |\det J\tau|\rangle 
\quad \mbox{for}\quad \psi\in \cs .
\end{equation}
It is continuous $\cs'\to\cs'$ as the adjoint of the 
continuous map $\psi\mapsto\psi\circ\tau|\det J\tau|$ on $\cs$: since 
$|\det J\tau|$ is in $C_{L_\infty}^\infty$, continuity on $\cs$ can be shown
using the higher-order chain rule to estimate each seminorm $q_{N,\alpha}(\psi\circ\tau)$,
cf.\ \eqref{eq:seminormsS_Nalpha}, by 
$\sum_{|\beta|\le|\alpha|}q_{N,\beta}(\psi)$
(changing variables, $\langle \sigma(\cdot) \rangle$ can be estimated using the Mean Value
Theorem on each $\sigma_j$).

We need a few further conditions, due to the anisotropic situation: one can neither expect 
$f\circ \sigma$ to have the same regularity as $f$,  e.g.\ if $\sigma $
is a rotation; nor that $f\circ \sigma \in L_{\vec{p}}$ when $f\in L_{\vec{p}}$.
On these grounds we first restrict to the situation in which
\begin{equation}\label{eq:conditionsOn_ap}
  a_0:=a_1= a_2 = \ldots  = a_{n-1}, \qquad p_0:= p_1 = \ldots =p_{n-1}
\end{equation}
and
\begin{equation}\label{eq:conditionsOn_sigma}
\sigma (x) = (\sigma'(x_1, \ldots , x_{n-1}),x_n) 
\quad \mbox{for all}\quad x\in \Rn .
\end{equation}

To prepare for Theorem~\ref{coo1} below, which gives sufficient conditions for
the invariance of $F^{s,\vec a}_{\vec p,q}$ 
under bounded diffeomorphisms of the type \eqref{eq:conditionsOn_sigma}, we first show that it
suffices to have invariance for sufficiently large $s$:

\begin{prop}
  \label{prop:invarianceAll}
Let $\sigma$ be a bounded diffeomorphism on $\Rn$
on the form in~\eqref{eq:conditionsOn_sigma}. When \eqref{eq:conditionsOn_ap} holds and
there exists $s_1\in\R$ with the property that $f\mapsto f\circ\sigma$ is a linear
homeomorphism of~$F^{s,\vec a}_{\vec p,q}(\Rn)$ onto itself for every $s>s_1$, 
then this holds true for all $s\in\R$.
\end{prop}

\begin{proof}
It suffices to prove for $s\le s_1$ that
\begin{equation}\label{eq:invarianceAllSufficient}
 \| \,  f\circ \sigma \, |F^{s,\vec{a}}_{\vec{p},q}  \| \le c\, 
 \| \,  f\, |F^{s,\vec{a}}_{\vec{p},q}  \|
\end{equation}
with some constant independent of $f$, as the reverse inequality  
then follows from the fact that the inverse of $\sigma$ is also a
bounded diffeomorphism with the structure in \eqref{eq:conditionsOn_sigma}. 
 
First $r>s_1-s+2a_n$ is chosen such that $d_0:=\frac r{2a_0}$ is a natural
number. Setting $d_n=\frac r{2a_n}$ and  
taking $\mu\in[0,1[\,$ such that $d_n-\mu\in\N$,
we have that $r_\mu := r-2\mu a_n > s_1-s$.

Now Lemma~\ref{lift} yields the existence of $h\in F^{s+r,\vec{a}}_{\vec{p},q}$ such that
$f = \Lambda_r h$, i.e.
\begin{alignat}{1}\label{eq:liftAppliedtof}
f = (1- \partial^2_{x_n})^{d_n-\mu}(1-\partial^2_{x_n})^\mu h 
   + \sum_{k=1}^{n-1}  (1- \partial^2_{x_k})^{d_0} h.
\end{alignat}
Setting $g_1=\big((1-\partial^2_{x_n})^\mu h\big)\circ\sigma$ and $g_0=h\circ\sigma$,
we may apply the higher-order chain rule to e.g.\ $h=g_0\circ\tau$
(using denseness of $\cs$ in $\cs'$ and the $\cs'$-continuity of composition 
in \eqref{eq:defCompositionSprime}, 
Appendix~\ref{app:higherOrderCR} extends to $\cs'$).
Taking into account that $\tau(x) = (\tau'(x'),x_n)$,
and letting prime indicate summation over multi-indices with $\beta_n=0$,
\begin{alignat}{1}
  \label{eq:applyingStructureTau}
f =  \sum_{l = 0}^{d_n-\mu} \eta_{n,l} \, \partial_{x_n}^{2l} g_1\circ\tau +
\sum_{k=1}^{n-1} \sideset{}{'}\sum_{|\beta|\leq 2d_0} \eta_{k,\beta} \, \partial^\beta g_0 \circ \tau ,
\end{alignat}
where $\eta_{n,l}:=(-1)^l \binom{d_n-\mu}{l}$ and the $\eta_{k,\beta}$
are functions containing derivatives at least  
of order 1 of $\tau$, and these can be estimated, say by $c
\prod_{1\leq m\leq 2d_0} \langle \partial_{x_k}^m
\tau\rangle^{2d_0}$. 
Composing with $\sigma$ and applying Lemma~\ref{derivative}(i) gives
for $d:=\min(1,q,p_0,p_n)$, when $\|\cdot\|$ denotes the $F^{s,\vec a}_{\vec p,q}$-norm,
\begin{gather}
\begin{split}
\label{eq:roundingOf}
\| \, f\circ \sigma \, \|^d   &\le 
\sum_{l = 0}^{d_n-\mu} \, |\eta_{n,l}|^d
\, \|\, \partial_{x_n}^{2l} g_1 \, \|^d
+ \, \sum_{k=1}^{n-1}  \sideset{}{'}\sum_{|\beta|\leq 2d_0} \| \, \eta_{k,\beta}\circ\sigma \, | M(F^{s,\vec a}_{\vec p,q})\|^d
\,  \|\, \partial^\beta g_0 \, \|^d
\\
& \le c\, \|\, g_1\, |F^{s+r_\mu,\vec{a}}_{\vec{p},q} \|^d   
 + 
\|\, g_0\, |F^{s+r,\vec{a}}_{\vec{p},q} \|^d
\sum_{k=1}^{n-1} \sideset{}{'}\sum_{|\beta|\leq 2d_0} \| \, \eta_{k,\beta}\circ\sigma \, |M(F^{s,\vec a}_{\vec p,q})\|^d.
\end{split}
\end{gather}
According to the remark preceding Lemma~\ref{mult}, the last sum is finite because 
$\eta_{k,\beta}\in C^\infty_{L_\infty}$.
Finally, since $s+r_\mu>s_1$ and $s+r>s_1$, the stated assumption means that
$h\mapsto g_1$ and $h\mapsto g_0$ are bounded, which in view of $r_\mu+2\mu a_n=r$ and
Lemmas~\ref{lift}--\ref{lem:lift} yields
\begin{equation}
  \| \, f\circ \sigma \, \|^d 
  \le c \|\, h \, |F^{s+r_\mu+2\mu a_n,\vec{a}}_{\vec{p},q} \|^d
      + \|\, h \, |F^{s+r,\vec{a}}_{\vec{p},q} \|^d
  \le c \|\, f \, |F^{s,\vec a}_{\vec p,q} \|^d;
\end{equation}
proving the boundedness of $f \mapsto f\circ \sigma$ in $F^{s,\vec a}_{\vec p,q}$ for all $s\in\R$.
\end{proof}

In addition to the reduction in Proposition~\ref{prop:invarianceAll}, 
we adopt in Theorem~\ref{coo1} below 
the strategy for the isotropic, unmixed case developed by Triebel~\cite[4.3.2]{T92}, who used Taylor
expansions for the inner and outer functions for large $s$.

While his explanation was rather sketchy, our task is to account for the fact that the
strategy extends to anisotropies and to mixed norms. Hence we give full details.
This will also allow us to give brief proofs of additional results in
Sections~\ref{group-ssect} and \ref{derived-sect} below.

To control the Taylor expansions, it will be crucial for us to exploit both the local means
recalled in  
Theorem~\ref{local} and the parameter-dependent set-up in Theorem~\ref{maxmod}. 
This is prepared for with the following discussion.

The functions $k_0$ and $k$ in Theorem~\ref{local} are for the proof of Theorem~\ref{coo1}
chosen (as we may) so that $N$ in the definition of $k$
fulfils $s<2N\underline{a}$ and so that both are \emph{even} functions and
\begin{equation}
  \label{eq:suppk0k}
  \supp k_0, \enskip \supp k \subset \big\{\,x\in\Rn \, \big|\, |x| \le 1\big\}. 
\end{equation}
The set $\Theta$ in Theorem~\ref{maxmod} is chosen to be
the set of $(n-1)\times (n-1)$ matrices $\cb = (b_{i,k})$ that, in terms of
the constants $c_\sigma, C_{\alpha,\sigma}$ in \eqref{det} and \eqref{det0},  respectively,
satisfy
\begin{gather}
  |\det \cb| \ge c_\sigma , 
  \label{eq:requirement1I}
\\
  \max_{i,k} |b_{i,k}| \le \max_{|\alpha|=1} C_{\alpha,\sigma}=: C_\sigma .
  \label{eq:requirement2I} 
\end{gather}

Splitting $z=(z',z_n)$, we set $g(z)={z'}^{\gamma'}k(z)$ for some $\gamma'\in\N_0^{n-1}$ 
(chosen later) and define
\begin{equation}\label{eq:psitheta}
  \psi_\theta(y)=g(\ca y',y_n)
\end{equation}
where $\theta$ is identified with $\ca^{-1}:=J\sigma'(x')$,
which obviously belongs to $\Theta$ (for each $x'$).

To verify that the above functions $\psi_\theta$, $\theta\in \Theta$, satisfy the moment condition 
\eqref{moment} for an $M_{\psi_\theta}$ such that the assumption $s<(M_{\psi_\theta}+1)\underline{a}$ 
in Theorem~\ref{maxmod} is fulfilled, note that
\begin{equation} 
  \widehat{\psi}_\theta(\xi)
  = |\det \ca|^{-1} \cf g \big(\tensor*[^t]{\ca}{^{-1}}\xi',\xi_n\big).
\end{equation}
Hence $D^\alpha \widehat{\psi}_\theta$ vanishes at $\xi=0$ when
$D^\alpha \widehat{g} = D^\alpha (-D_{\xi'})^{\gamma'} \widehat{k}(\xi)$ does so.  
As $\widehat{k}(\xi)=-|\xi|^{2N}\widehat{k^0}(\xi)$ and $\widehat{k^0}(0)\neq 0$, 
we have $D^\alpha \widehat{g}(0)=0$ for $\alpha$ satisfying $|\alpha|+|\gamma'|\leq 2N-1$. 
In the course of the proof below, cf.
Step 3, we obtain a $\theta$-independent estimate of $|\gamma'|$, hence of $M_{\psi_\theta}$.

Moreover, the constant $A$ in Theorem~\ref{maxmod} is finite:
Basic properties of the Fourier transform give the following estimate, where
the constant is independent of $\ca^{-1}\in \Theta$:
\begin{equation}
  \begin{split}
\| \, D^\alpha \cf \psi_\theta\, | L_\infty\|
& \leq   \int  |y^\alpha g (\ca y',y_n)|\, dy\\
& =   |\det \ca^{-1}|
\int  |z_n^{\alpha_n}| \, |(\ca^{-1} z')^{\alpha'} | \, | g (z)|\, dz
 \le  c(\alpha,C_\sigma) 
\int_{|z|\le 1}  |k (z)|\, dz .
\end{split}
\end{equation}
To estimate $B$ we exploit that $\cf\colon B^{n/2}_{2,1}(\Rn)\to L_1(\Rn)$ is bounded according
to Szasz's inequality (cf.~\cite[Prop. 1.7.5]{ScTr87}) and obtain
\begin{equation}
\| \, (1+ |\cdot|)^{M+1}\, D^\gamma \cf \psi_\theta \, | L_1\|
 \le  c\, \|\, y^\gamma  g (\ca y',y_n)\, | B^{M+1+\frac{n}2}_{2,1}\|
 \le  c(\gamma,C_\sigma,C_\tau)\, \|\, k\, | C^m_0\|,
\end{equation}
when $m\in \N$ is chosen so large that $m>M+1+n/2$. In fact, the last inequality is obtained 
using the embeddings $C_0^m \hookrightarrow H^m \hookrightarrow B^{M+1+n/2}_{2,1}$ 
and the estimate
\begin{equation}
  \|\, y^\gamma \psi_\theta\, | C^m_0 \|
  = \sup
  \big|\, \partial^\alpha \big( y^\gamma (\ca y')^{\gamma'} k(\ca y',y_n) \big) \big|
  \leq c(\gamma,C_\sigma,C_\tau)\, \|\, k\, | C^m_0\|.
\end{equation}
This relies on the higher-order chain rule, cf.~Appendix~\ref{app:higherOrderCR}, and
the support of $k$: it suffices to use the supremum over $|\alpha|\le m$ and
$\{y\in\R^n\, |\, |\ca y'|^2 + y_n^2 \leq 1\}$, and for a point in this set
$|y'| \leq \| \ca^{-1}\| | \ca y'| \leq c(C_\sigma)$, 
so we need only estimate on an $\ca$-independent cylinder.

Replacing $k$ by $k_0$ in the definition of $g$
and setting $\psi_{\theta,0}(y):=g(\ca y',y_n)$, the finiteness
of $C$ and $D$ follows analogously. The Tauberian properties follow from
$\int k_0\ne0\ne\int k^0$. 

Hence all assumptions in Theorem~\ref{maxmod} are satisfied, and we are thus ready to prove
our main result

\begin{thm}\label{coo1}
If $\sigma$ is a bounded diffeomorphism on $\Rn$
on the form in~\eqref{eq:conditionsOn_sigma},
then $f \mapsto f \circ \sigma$ is a linear homeomorphism 
$F^{s,\vec{a}}_{\vec{p},q}(\Rn)\to F^{s,\vec{a}}_{\vec{p},q}(\Rn)$ for all $s\in\R$ when
\eqref{eq:conditionsOn_ap} holds.
\end{thm}

\begin{proof}
According to Proposition~\ref{prop:invarianceAll}, it suffices to consider $s>s_1$, say for
\begin{equation}
\label{eq:defOfs1}
s_1 := K_0 a_0+(n-1) \frac{a_0}{p_0} + \frac{a_n}{\min(p_0,p_n)},
\end{equation}
whereby $K_0$ is the smallest integer satisfying
\begin{equation}
\label{eq:defOfK0}
K_0 a_0 >  (n-1) \frac{a_0}{p_0} + \frac{a_n}{\min(p_0,p_n)}.
\end{equation}
We now let $s\in\, ]s_1,\infty[\,$ be given and take some $K\geq K_0$, i.e.~$K$
solving \eqref{eq:defOfK0}, such that
\begin{equation}\label{ka}
  K a_0 + (n-1) \frac{a_0}{p_0} + \frac{a_n}{\min (p_0,p_n)}  < s < 2 K a_0.
\end{equation}
(The interval thus defined is non-empty by \eqref{eq:defOfK0}, and the 
left end point is at least $s_1$.) 

Note that \eqref{ka} yields that every $f \in F^{s,\vec{a}}_{\vec{p},q}$ is continuous, 
cf.~Lemma~\ref{derivative}(iii); so are even the derivatives
$D^\beta f$ for $\beta=(\beta_1, \ldots, \beta_{n-1},0)$, $|\beta|\le K$, since 
$s-\beta \cdot \vec{a} =  s - |\beta| a_0 > \vec a\cdot 1/{\vec p}$.

\smallskip{\em Step 1.} 
For the norms $\|\,f\circ\sigma\,|F^{s,\vec a}_{\vec p,q}\|$ and  
$\|\,f\,|F^{s,\vec a}_{\vec p,q}\|$ in inequality \eqref{eq:invarianceAllSufficient}, which also here
suffices, we use
Theorem~\ref{local} with $2N>(K-1)(2K-1)+s/\underline{a}$.

By the symmetry of $k_0$ and $k$ in \eqref{eq:suppk0k}, we shall estimate
\begin{equation}\label{eq:convolution_kj_fsigma}
k_j * (f\circ \sigma) (x)=  \int_{|z|\leq 1} k(z)\, f(\sigma(x + 2^{-j\vec{a}} z))\, dz  , 
\quad j \in \N,
\end{equation}
together with the corresponding expression for $k_0$, where $k$ is replaced by $k_0$.

First we make a Taylor expansion of the entries in
$\sigma'(x'):=(\sigma_1(x'),\ldots,\sigma_{n-1}(x'))$, to the order $2K-1$.
So for $\ell=1,\ldots,n-1$ there exists $\omega_\ell\in\, ]0,1[$ such that
\begin{alignat}{1}
  \sigma_\ell (x' + z')  =  
  \sum_{|\alpha'|< 2K} \frac{\partial^{\alpha'} \sigma_\ell(x')}{\alpha'!} z'^{\,\alpha'}
  + \sum_{|\alpha'| = 2K} \frac{\partial^{\alpha'} \sigma_\ell (x'+\omega_\ell  z')}{\alpha' !}
  {z'}^{\,\alpha'} .  
  \label{eq:sigmaTaylor}
\end{alignat}
For convenience, we let $\sum'_\alpha$ denote summation over multiindices $\alpha\in\N_0^n$
having $\alpha_n=0$ and define the vector of Taylor polynomials, respectively entries of a
remainder $R$,  
\begin{equation}
  P_{2K-1}(z') = 
  \sideset{}{'}\sum_{|\alpha|\leq 2K-1} \frac{\partial^{\alpha} \sigma'(x')}{\alpha!}{z}^{\alpha},
\qquad
  R_\ell (z') =  \sideset{}{'}\sum_{|\alpha| = 2K} \frac{\partial^{\alpha} \sigma_\ell (x' + 
  \omega_\ell z')}{\alpha !}\, {z}^{\alpha}. 
\end{equation}

Applying the Mean Value Theorem to $f$, cf.~\eqref{ka}, now yields an $\widetilde \omega\in\,
]0,1[\,$ so that
\begin{gather}
\begin{split}\label{eq-9}
|k_j * (f\circ \sigma) (x)|  \le 
\Big| &\int_{|z|\leq 1} k(z)\, f(P_{2K-1}(2^{-ja'}z'), x_n + 2^{-ja_n}z_n)\, dz\Big|
\\ 
& +
\sum_{d=1}^{n-1} \int_{|z|\leq 1} |k (z) \, \partial_{x_d} f(y',x_n+ 2^{-ja_n}z_n)\, R_d(2^{-ja'}z')|\, dz,
\end{split}
\end{gather}
when $y':=P_{2K-1}(2^{-ja'}z')+\widetilde\omega (R_1(2^{-ja'}z'),\ldots,R_{n-1}(2^{-ja'}z'))$.
Using \eqref{det0} and \eqref{eq:sigmaTaylor}, it is obvious that this $y'$ fulfils
\begin{equation}\label{dist}
|\sigma(x)-(y',x_n+2^{-j a_n} z_n)| \leq |\sigma'(x') -y'| + |2^{-j a_n} z_n| < C
\end{equation}
for each $z\in\supp k$ and some constant $C$ depending only on $n$
and $C_{\alpha,\sigma}$ with $|\alpha|\le 2K$.

\smallskip
{\em Step 2}.
Concerning the remainder terms in \eqref{eq-9} we exploit \eqref{dist} to get
\begin{equation}
  \begin{split}
  \int_{|z|\leq 1} |k (z) \, &\partial_{x_d} f(y',x_n + 2^{-ja_n}z_n)\, R_d(2^{-ja'}z')|\, dz
\\
 & \le  2^{-2jKa_0} \Big(\sum_{|\alpha'| = 2K} \frac{\| \partial^{\alpha'} \sigma_d \, |
   L_\infty\|}{\alpha'!}\Big) 
 \int_{|z|\leq 1} |k(z)|\, dz  \sup_{|\sigma(x)-y|<C}| \partial_{x_d} f(y)|  .
\end{split}
\end{equation}
The exponent in $2^{-2jKa_0}$ is a result of  \eqref{eq:conditionsOn_ap}
and the chosen Taylor expansion of $\sigma(x+2^{-j\vec{a}}z)$, and since $s-2Ka_0<0$ the norm of
$\ell_q$ is trivial to calculate, whence 
\begin{gather}
\begin{split}
\label{eq:step1beforeCC}
\Big\|\, 2^{js}  \int_{|z|\leq 1} |k(z) \,  \partial_{x_d} f(y',x_n+ 2^{-ja_n}z_n)\,
&R_d(2^{-ja'}z')|\, dz \, \Big| L_{\vec{p}}(\ell_q )\Big\|\\
&\le  c\,  \Big\|\,  \sup_{|\sigma (x)-y|< C} | \partial_{x_d} f(y)|\,  
\Big| L_{\vec{p}}(\mathbb{R}^n_x) \Big\| .
\end{split}
\end{gather}
Now we use that $p_1= \ldots =p_{n-1}$ to change variables in the resulting integral over $\R^{n-1}$,
with $\tau'$ denoting $(\sigma')^{-1}$.
Since Lemma~\ref{lem:lpLTEstimate} in view of \eqref{ka} applies to $\partial_{x_d} f$, $d=1,\ldots, n-1$, the right-hand side of the last
inequality can be estimated, using also Lemma~\ref{derivative}(i), by
\begin{equation}
\label{eq:endSubstep11}
  c\, \Big(\sup_{y\in {\R}^{n-1}}  |\det J\tau'(y)|\Big)^{1/p_0}
 \,  
 \|\, \partial_{x_d} f \, |  F^{s-a_d, \vec{a}}_{\vec{p},q}\|
 \leq c\, \|\, f\,| F^{s,\vec a}_{\vec p,q}\|.
\end{equation}

\smallskip
{\em Step 3}. To treat the first term in \eqref{eq-9}, we  
Taylor expand $f(\cdot,x_n)$, which is in $C^K(\mathbb{R}^{n-1})$.
Setting $P(z')=P_{2K-1}(z')-P_1(z')$,  expansion at the vector $P_1(2^{-ja'}z')$ gives 
\begin{alignat}{1}
  \nonumber
 f(P_{2K-1}(2^{-ja'}z'), x_n+ 2^{-ja_n}z_n)
  =
 \sideset{}{'}\sum_{0 \le |\beta| \le K-1} &\frac{D^\beta 
 f (P_1(2^{-ja'}z'), x_n+ 2^{-ja_n}z_n)}{\beta !} \, P(2^{-ja'}z')^\beta 
 \\
  +  
  \sideset{}{'}\sum_{|\beta| = K} &\frac{D^\beta 
 f (y', x_n+ 2^{-ja_n}z_n)}{\beta !} \, P(2^{-ja'}z')^\beta  ,
 \label{eq:taylor_f_orderK-1}
\end{alignat}
where  $y'$ is a vector analogous to that in \eqref{eq-9} and satisfies
\eqref{dist}, perhaps with another $C$. 

To deal with the remainder in \eqref{eq:taylor_f_orderK-1},
note that the order was chosen to ensure
that, in the powers $P(2^{-ja'}z')^\beta$, the $l$'th factor
is the $\beta_l$'th power of a sum of terms each containing a factor $2^{-ja_0|\alpha'|}$
with $|\alpha'|\geq 2$. Hence each $|\beta|=K$ in total contributes by $O(2^{-2jKa_0})$.
More precisely, as in Step~2 we obtain
\begin{equation}
  \begin{split}
\int_{|z|\leq 1} &\Big| k(z)\, \sideset{}{'}\sum_{|\beta| = K} \frac{D^\beta 
 f (y', x_n+ 2^{-ja_n}z_n)}{\beta !}
  \,P(2^{-ja'}z')^\beta \Big| \, dz
\\
&\qquad \le  2^{-j2Ka_0} \int_{|z|\leq 1} |k(z)|\, dz \,
 \big(\sum_{2\le |\alpha| \le 2 K-1} C_{\alpha,\sigma} \big)^{K}
\,\sideset{}{'}\sum_{|\beta|=K}\sup_{|\sigma (x)-y|< C}  | D^\beta f(y)|.
\end{split}
\end{equation}
In view of \eqref{ka}, Lemma~\ref{lem:lpLTEstimate} barely also applies to 
$D^\beta f$ for $|\beta|=K$, so the above gives
\begin{equation}
  \begin{split}
\Big\| \, 2^{sj} \int_{|z|\leq 1} k(z)& \sideset{}{'}\sum_{|\beta| = K} \frac{D^\beta 
f (y', x_n+ 2^{-ja_n}z_n)}{\beta !}
\,P(2^{-ja'}z')^\beta \, dz\,\Big| L_{\vec{p}}(\ell_q)\Big\|
\\
\quad &\le   c\, \big(\sup_{y\in {\R}^{n-1}} |\det J\tau'(y)|\big)^{\frac1{p_0}}
\sideset{}{'}\sum_{|\beta|=K}  
 \|\, D^\beta f \, | F^{s-\beta \cdot \vec{a}, \vec{a}}_{\vec{p},q}\| 
\le c\, \|\, f\, | F^{s,\vec a}_{\vec p,q}\| .
\end{split}  
\end{equation}

Now it remains to estimate the other terms resulting from \eqref{eq:taylor_f_orderK-1}, i.e.
\begin{equation}
 \sideset{}{'}\sum_{0 \le |\beta| \le K-1} \int_{|z|\leq 1}  k(z)\, 
  \frac{D^\beta  f (P_1(2^{-ja'}z'), x_n+ 2^{-ja_n}\,z_n)}{\beta !}
  \, P(2^{-ja'}z')^\beta\, dz.  
  \end{equation}
Using the multinomial formula on
$P(z') = \sideset{}{'}\sum_{2\le|\gamma|\leq 2K-1} {z}^{\gamma}\partial^{\gamma}\sigma'(x')/\gamma!$
and the $g$ and $\psi_\theta$ discussed in \eqref{eq:psitheta}, the above task is finally reduced
to controlling terms like 
\begin{gather}
\begin{split}
  I_{j,\beta,\gamma} (\sigma'(x'),x_n) &:= 2^{-2j|\beta|a_0} \int_{|z|\leq 1} g(z)
  D^\beta f (\sigma'(x')+2^{-ja_0} J\sigma'(x') z', x_n+ 2^{-ja_n}z_n)\,  dz
\\
  &= 2^{-2j|\beta|a_0} |\det \ca|  
  \int \psi_{\theta}(y)
  D^\beta f (\sigma'(x') +2^{-ja_0} y' , x_n+ 2^{-ja_n} y_n) \,  dy.
\end{split}\label{eq:Ij_coordinateChange}
\end{gather}
Note that in $g$, $\psi_\theta$ we have $2 \le |\gamma| \le |\beta|(2K-1)$ and 
$|\beta|\le K-1$, $\beta_n=0=\gamma_n$. 

\smallskip
{\em Step 4}.
Before we estimate \eqref{eq:Ij_coordinateChange}, 
it is first observed that all previous steps apply in a similar way to the convolution
$k_0 * (f\circ \sigma)$\,---\,except in this case there is no dilation, so 
the $\ell_q$-norm is omitted and the function $\psi_\theta$ is replaced by $\psi_{\theta,0}$.

So, when collecting the terms of the form \eqref{eq:Ij_coordinateChange} with finitely many $\beta$,
$\gamma$ in both cases
(omitting remainders from Steps~2--3), we obtain with two changes of variables and \eqref{*max-ineq},
\begin{equation}
  \begin{split}
\Big\|\, &\sideset{}{'}\sum_{\beta, \gamma} I_{0,\beta, \gamma}(\sigma'(x'),x_n)\, \Big| L_{\vec{p}} \Big\| +
\Big\|\, 2^{js} \sideset{}{'}\sum_{\beta, \gamma}  I_{j,\beta,\gamma}(\sigma'(x'),x_n) \, \Big|
L_{\vec{p}}(\ell_q) \Big\|
\\
& \le c \sideset{}{'}\sum_{\beta, \gamma} 
\big(\sup_{y\in \mathbb{R}^{n-1}} |\det J\tau'(y)|\big)^{\frac1{p_0}} 
\Big(\, \Big\| \int \psi_{\theta,0}(y) D^\beta f(x-y)\, dy \, \Big| L_{\vec{p}}\Big\| 
 \\
&\quad\qquad\qquad\qquad\qquad + \Big\|\, 2^{j(s-2|\beta|a_0)} \int \psi_\theta(y) D^\beta f(x-2^{-j\vec{a}}y) \, dy  \, \Big| L_{\vec{p}}(\ell_q) \Big\|\, \Big)
 \\
& \le c \sideset{}{'} \sum_{\beta, \gamma}
\Big\|\, \Big\{ 2^{j(s-2|\beta|a_0)} \sup_{\theta\in \Theta} \, \psi_{\theta,j}^* D^\beta f \Big\}_{j=0}^\infty \, 
\Big| L_{\vec{p}}(\ell_q) \Big\| . \label{eq:justBeforeTh2}
\end{split}
\end{equation}
Here we apply Theorem~\ref{maxmod} to the family of functions $\psi_{\theta,0}, \psi_\theta$  
with the $\varphi_j$ chosen as the Fourier transformed of the 
system in the Littlewood-Paley decomposition, cf.~\eqref{unity}.
Estimating $|\gamma|$, the $\psi_\theta$ satisfy the moment condition \eqref{moment} with 
$M_{\psi_\theta}:=2N-1-(K-1)(2K-1)$, which fulfils $s<(M_{\psi_\theta}+1)\underline{a}$, because
of the choice of $N$ in Step~1. So, by applying Theorem~\ref{inverse}  
and Lemma~\ref{derivative}(i), using $s-2|\beta|a_0\le s- \beta \cdot \vec{a}$,
the above is estimated thus:
\begin{equation}
  \begin{split}
  \Big\|\, \Big\{ 
                 &2^{js} \sideset{}{'}\sum_{\beta, \gamma} I_{j,\beta,\gamma}(\sigma'(x'),x_n) 
   \Big\}_{j=0}^\infty \, \Big| L_{\vec{p}}(\ell_q) \Big\|
\\
  &\quad\le c(A+B+C+D)\sideset{}{'}\sum_{\beta,\ \gamma}
   \left\| \, \left\{ 
    2^{j(s-2|\beta|a_0)} (\cfi\Phi_j)^* D^\beta f \right\}_{j=0}^\infty \, 
   \middle|  L_{\vec{p}}(\ell_q)\right\|
\\
  &\quad \le c\sideset{}{'}\sum_{\beta,\ \gamma}
  \|\, D^\beta f\, |\, F^{s-2|\beta|a_0,\vec{a}}_{\vec{p},q} \, \|\, 
   \le c\, \| \,  f\, |F^{s,\vec{a}}_{\vec{p},q}  \| .
\end{split}
\end{equation}
This proves the necessary estimate for the given $s>s_1$.
\end{proof}

\subsection{Groups of bounded diffeomorphisms}
  \label{group-ssect}
It is not difficult to see that the proofs in Section~\ref{bDiff-ssect} 
did not really use that $x_n$ is a single variable.
It could just as well have been replaced by a whole group of variables $x''$, corresponding to a
splitting $x=(x',x'')$, provided $\sigma$ acts as the identity on $x''$. 

Moreover,  
$x'$ could equally well have been `embedded' into $x''$, that is $x''$ could contain variables
$x_k$ both with $k<j_0$ and with $k>j_1$ when $x'=(x_{j_0},\ldots,x_{j_1})$ 
(but no interlacing); in particular the changes of variables yielding 
\eqref{eq:endSubstep11} would carry over to this situation when $p_{j_0}=\ldots=p_{j_1}$.
It is also not difficult to see that Proposition~\ref{prop:invarianceAll} 
extends to this situation when
$a_{j_0}=\ldots=a_{j_1}$ (perhaps with several $g_1$-terms, each having a value of $\mu$). 

Thus we may generalise Theorem~\ref{coo1} to situations with a splitting into $m\ge2$ groups, 
i.e.\ $\Rn=\R^{N_1}\times\ldots\times\R^{N_m}$ where $ N_1 + \ldots + N_m=n$, namely when
\begin{align}
  \vec{p} &= (\underbrace{p_1, \ldots, p_1}_{N_1}, \underbrace{p_{2} , \ldots , p_2}_{N_2}, 
  \ldots ,\underbrace{p_m, \ldots, p_m}_{N_m}), 
\label{group1} \\
  \vec{a} &= ({a_1, \ldots, a_1}, {a_{2} , \ldots , a_2}, 
  \ldots ,{a_m, \ldots, a_m}),
\label{group2} \\
  \sigma(x)&=(\sigma'_1(x_{(1)}), \ldots , \sigma'_m(x_{(m)}))
\label{group3}
\end{align}
with arbitrary bounded diffeomorphisms  $\sigma'_j$ on $\R^{N_j}$ and $x_{(j)}\in\R^{N_j}$.

Indeed, viewing $\sigma$ as a composition of $\sigma_1:=\sigma'_1\otimes 
\operatorname{id}_{\R^{n-N_1}}$ etc.\ on $\Rn$, the above gives
\begin{equation}
\|\, f\circ\sigma\, | F^{s,\vec a}_{\vec p,q}\| \leq c\, \|\, f\circ\sigma_m\circ\ldots\circ\sigma_2\,
| F^{s,\vec a}_{\vec p,q}\|\leq \ldots\leq c\,\|\, f\, | F^{s,\vec a}_{\vec p,q}\|. 
\end{equation}

\begin{thm}
  \label{coo2-thm}
 $f\mapsto f\circ\sigma$ is a linear homeomorphism on $F^{s,\vec a}_{\vec p,q}$ 
when \eqref{group1}, \eqref{group2}, \eqref{group3} hold.
\end{thm}

\section{Derived results} 
  \label{derived-sect}

\subsection{Diffeomorphisms on Domains}
The strategies of Proposition~\ref{prop:invarianceAll} and Theorem~\ref{coo1} also give the
following local version. E.g., for the paraboloid $U=\{\,x\mid x_n>x_1^2+\ldots+x_{n-1}^2\,\}$ 
we may take $\sigma$ to consist in a rotation around the $x_n$-axis; cf.\
\eqref{eq:conditionsOn_sigma}. 

\begin{thm}\label{thm:localVersion}
Let $U,V\subset \Rn$ be open and $\sigma: U\to V$ a $C^\infty$-bijection 
as in~\eqref{eq:conditionsOn_sigma}. 
If \eqref{eq:conditionsOn_ap} is fulfilled and $f\in F^{s,\vec a}_{\vec p,q}(V)$ has compact support,
then $f \circ \sigma \in F^{s,\vec a}_{\vec p,q}(U)$ and 
\begin{equation}
  \label{eq:localVersion}
  \|\, f\circ \sigma\, | F^{s,\vec a}_{\vec p,q}(U)\| 
  \leq c \|\, f\, | F^{s,\vec a}_{\vec p,q}(V)\|
\end{equation}
holds for a constant $c$ depending only on $\sigma$ and the set $\supp f$.
\end{thm}

\begin{proof}
{\em Step 1.}
Let us consider $s>s_1$, cf.~\eqref{eq:defOfs1}, and adapt the proof of Theorem~\ref{coo1} to
the local set-up.  We shall prove the statement for the
$f\in F^{s,\vec a}_{\vec p,q}(V)$ satisfying $\supp f\subset K\subset V$ for some arbitrary compact set
$K$. First we fix $r\in\,]0,1[\,$ so small that
\begin{equation}
  6r<\min\big(\dist(K,\R^{n}\setminus V),\,
            \dist({\sigma}^{-1}(K),\R^{n}\setminus U)\big).
\label{6r-ineq}
\end{equation}
Then, by Lemma~\ref{lem:equalityInfnorm}, we have
$\|\, f\circ\sigma\, | F^{s,\vec a}_{\vec p,q}(U)\| = \|\, e_U(f\circ\sigma)\, | F^{s,\vec a}_{\vec
  p,q}\|$ 
when Theorem~\ref{local} is utilised for $k_0,k\in\cs$, say so that 
$\supp k_0,\supp k \subset B(0,r)$; cf.\ also \eqref{eq:suppk0k}.
Extension by 0 outside $U$ of $f\circ\sigma$
is redundant, for it suffices to integrate over $x\in W:=\supp (f\circ\sigma)+\overline{B}(0,r)$.
However, to apply the Mean Value Theorem, cf.~\eqref{eq-9}, 
we extend $f$ by 0 instead, i.e.~we consider 
\eqref{eq:convolution_kj_fsigma} with integration over $|z|\leq r$ and with $f$ replaced by $e_V f$.

Since $e_V f$ inherits the regularity of $f$ (cf.\  Lemma~\ref{lem:equalityInfnorm})
and $\partial^\alpha \sigma$ can be estimated on the compact set ${W}$, 
the proof of Theorem~\ref{coo1} carries over straightforwardly. 
E.g.\ one obtains a variant of \eqref{eq:endSubstep11} where $|\det J\tau'(x')|^{1/p_0}$ is
estimated over  
$\{ x'\, | \exists x_n\colon (x',x_n)\in\sigma ({W})\}$,
and the integration is then extended to $\Rn$, which by Lemma~\ref{lem:equalityInfnorm} yields
\begin{equation}
\|\, \sup_{|x-y|<C} |\partial_{x_d} e_V f(y)|\, | L_{\vec p}(\mathbb{R}^n_x)\|
\leq c\, \|\, e_V f\, | F^{s,\vec a}_{\vec p,q}(\Rn)\|
= c\, \|\, f\, | F^{s,\vec a}_{\vec p,q}(V)\|.
\end{equation}

To estimate the first term in \eqref{eq-9} in this local version, the argumentation there is
modified as above  and the set $\Theta$ is chosen to be the set of all $(n-1)\times(n-1)$ matrices
satisfying~\eqref{eq:requirement1I} with infimum over $x\in {W}$ and
\eqref{eq:requirement2I} with 
$C_{\sigma} := \max_{1\le j\le n,\ |\alpha|=1 } \sup_{x\in{W}} |D^\alpha\sigma_j(x)|$.

Before applying Theorem~\ref{maxmod} to the new estimate \eqref{eq:justBeforeTh2}, 
the integration is extended  to $\Rn$ (using $e_V f$). Then application
of Theorem~\ref{maxmod} and Theorem~\ref{inverse} together with  
Lemma~\ref{lem:equalityInfnorm} finishes the proof for $s>s_1$.

\smallskip
\noindent{\em Step 2.}
For $s\le s_1$ we use Lemma~\ref{lift} to write $e_V f=\Lambda_r h$ for some
$h\in F^{s+r,\vec a}_{\vec p,q}(\Rn)$;
hence the identity \eqref{eq:liftAppliedtof} holds in $\cd'(\Rn)$ for $e_V f$ and $h$.
Applying~$r_V$ to both sides and using that it commutes with differentiation on~$C_0^\infty$, 
hence on $\cd'$, we obtain \eqref{eq:applyingStructureTau} as an identity in $\cd'(V)$ for
the new $g_0:=(r_V h)\circ\sigma$ and $g_1:=\big(r_V (1-\partial_{x_n}^2)^\mu h\big)\circ\sigma$.

Composing with $\sigma$ yields an identity in $\cd'(U)$, when $\eta_{k,\beta}\circ\sigma$ is treated
using cut-off functions. E.g.\ we can take $\chi,\chi_1\in C_0^\infty(U)$ with $\chi\equiv 1$ on 
$\supp(f\circ\sigma)+\overline{B}(0,r)=:W_r$ and $\supp\chi\subset W_{2r}$,
while $\chi_1\equiv 1$ on $W_{3r}$ and $\supp\chi_1\subset W_{4r}$. This entails
\begin{equation}
  \label{eq:localVersionChi}
  \chi\cdot f\circ\sigma = 
  \sum_{l = 0}^{d_n-\mu} \eta_{n,l}\chi \, \partial_{x_n}^{2l} (\chi_1 g_1) +
  \sum_{k=1}^{n-1} \sideset{}{'}\sum_{|\beta|\leq 2d_0} \eta_{k,\beta}\circ\sigma \cdot\chi \, \partial^\beta (\chi_1 g_0).
\end{equation}
Using $e_U$ on both sides (and omitting $\Rn$ in the spaces),
Lemma~\ref{lem:equalityInfnorm} and Lemma~\ref{mult} imply
\begin{align}
  \|\, f\circ\sigma\, |F^{s,\vec a}_{\vec p,q}(U)\|^d
  \le c\sum_{l=0}^{d_n-\mu} \|\, e_U (\partial_{x_n}^{2l} (\chi_1g_1))\, |F^{s,\vec a}_{\vec p,q}\|^d
  + c \sideset{}{'}\sum_{|\beta|\leq 2d_0} \|\, e_U(\partial^\beta(\chi_1g_0))\, |F^{s,\vec a}_{\vec
    p,q}\|^d. 
\end{align}
As $e_U$ and differentiation commute on $\ce'(U)\ni\chi_1 g_j$,
Lemma~\ref{derivative}(i) leads to an estimate from above. But
Lemma~\ref{lem:equalityInfnorm} applies since the supports are in $W_{4r}$, so with
$\widetilde\chi_1:=\chi_1\circ\tau$ we find that the above is less than or equal to 
\begin{equation}
  \begin{split}
  c \|\, e_U(\chi_1 g_1)\, |F^{s+r_\mu,\vec a}_{\vec p,q} \|^d + 
  c \|\, e_U(\chi_1 g_0)\, |F^{s+r,\vec a}_{\vec p,q}&\|^d\\
  = c \big\|\, \big(\widetilde\chi_1\cdot r_V (1-\partial_{x_n}^2)^\mu h\big)\circ\sigma\,
  \big|F^{s+r_\mu,\vec a}_{\vec p,q}(U) &\big\|^d  
  + c \|\, (\widetilde\chi_1 \cdot r_V h)\circ\sigma\, |F^{s+r,\vec a}_{\vec p,q}(U)\|^d.
  \end{split}
\end{equation}
Using Step 1 and Lemma~\ref{mult}, Lemma~\ref{lem:lift}, Lemma~\ref{lift} and 
Lemma~\ref{lem:equalityInfnorm}, this entails
\begin{equation}
  \begin{split}
    \|\, f\circ\sigma\, |F^{s,\vec a}_{\vec p,q}(U)\|^d
&\le
  c \|\, (1-\partial_{x_n}^2)^\mu  h\, |F^{s+r_\mu,\vec a}_{\vec p,q}\|^d
  + c \|\, h \, |F^{s+r,\vec a}_{\vec p,q}\|^d
\\
& \le c \|\, \Lambda_r^{-1} e_V f\, |F^{s+r,\vec a}_{\vec p,q}\|^d
   \le c \|\, f\, |F^{s,\vec a}_{\vec p,q}(V)\|^d.
  \end{split}
\end{equation}
This shows the local theorem for $s\leq s_1$.
\end{proof}

There is also a local version of Theorem~\ref{coo2-thm}, with similar proof, namely

\begin{thm}
  \label{coo2_local-thm}
Let $\sigma_j: U_j\to V_j$, $j=1,\ldots,m$, be $C^\infty$ bijections,
where $U_j,V_j \subset \R^{N_j}$ are open. 
When $\vec a,\vec p$ fulfill \eqref{group1}--\eqref{group2}
and when $f\in F^{s,\vec a}_{\vec p,q}(U_1\times\dots\times U_m)$ has
compact support, then~\eqref{eq:localVersion} holds true for 
$U=U_1\times\dots\times U_m$ and $V=V_1\times\dots\times V_m$. 
\end{thm}

As a preparation for our coming work~\cite{JoMHSi2},
we include a natural extension to the case of an infinite cylinder, where $\supp f$ is only
required to be compact on cross sections:

\begin{thm}\label{thm:infiniteCylinder}
Let $\sigma: U\times\R\to V\times\R$ be a $C^\infty$-bijection on the
form in \eqref{eq:conditionsOn_sigma}, and $U,V\subset\R^{n-1}$ open.
If \eqref{eq:conditionsOn_ap} holds and
$f\in F^{s,\vec a}_{\vec p,q}(V\times\R)$ has $\supp f \subset K\times\R$, whereby $K\subset V$
is compact, then $f\circ\sigma\in F^{s,\vec a}_{\vec p,q}(U\times\R)$ and 
\begin{equation}
\|\, f\circ\sigma\, | F^{s,\vec a}_{\vec p,q}(U\times\R) \|
\leq c(\supp f,\sigma) \|\, f\, | F^{s,\vec a}_{\vec p,q}(\Rn) \|.
\end{equation}
\end{thm}

\begin{proof}
We adapt the proof of Theorem~\ref{thm:localVersion}:
in Step~1  we take $r\in\,]0,1[$ so small that
$6r$ is less than both $\dist(K,\R^{n-1}\setminus V)$ and $\dist({\sigma'}^{-1}(K),\R^{n-1}\setminus U)$.
Since the extension by zero $e_{V\times\R}f$
is well defined, as $K\subset V$ is compact, it is an immediate corollary to the proof of
Lemma~\ref{lem:equalityInfnorm} that
\begin{equation}
  \|\, f\, | F^{s,\vec a}_{\vec p,q}(V\times\R)\| = \|\, e_{V\times\R} f\, | F^{s,\vec a}_{\vec
    p,q}(\Rn)\|.
\label{eq:equalityInfNorm}
\end{equation}
Then the proof for $s>s_1$ follows that of Theorem~\ref{thm:localVersion}, with
$W:=(\sigma'^{-1}(K)+\overline{B}(0,r))\times\R$. 

For $s\le s_1$ we have $e_{V\times\R}f =\Lambda_r  h$ 
for some $h\in F^{s+r,\vec a}_{\vec p,q}(\Rn)$; cf.\ Lemma~\ref{lift}. 
Hence~\eqref{eq:applyingStructureTau} holds as an
identity in $\cd'(V\times\R)$ for $g_1:=\big(r_{V\times\R}(1-\partial_{x_{n+1}}^2)^\mu
h\big)\circ\sigma$ and $g_0:=(r_{V\times\R} h)\circ\sigma$. 

The $\eta_{k,\beta}\circ \sigma$ are controlled using cut-off functions  
$\chi,\chi_1\in C^\infty_{L_\infty}(U)$ with similar properties in terms of the sets
$W_r=(\sigma'^{-1}(K)+\overline{B}(0,r))\times\R$. Thus we obtain
\eqref{eq:localVersionChi} in $\cd'(U\times\R)$. 

Now, as in \eqref{eq:equalityInfNorm} it is seen that $f\circ\sigma$
and $e_{U\times\R} (\chi\cdot f\circ\sigma)$ have identical norms,
so the estimates in Step 2 of the proof of Theorem~\ref{thm:localVersion} finish the proof,
mutatis mutandis.
\end{proof}

\subsection{Isotropic Spaces}

Going to the other extreme, when also $a_n = a_0$ and $p_n=p_0$,
then the Lizorkin--Triebel spaces are invariant under any bounded diffeomorphism 
(i.e.\ without \eqref{eq:conditionsOn_sigma}), 
since in that case we can just change variables in all coordinates,
in particular in \eqref{eq:step1beforeCC}--\eqref{eq:endSubstep11}. Moreover, we can adapt 
Proposition~\ref{prop:invarianceAll} by taking $d_n=d_0$ and $\mu=0$ in the proof; and the set-up
prior to Theorem~\ref{coo1} is also easily modified to the isotropic situation. Hence we obtain 

\begin{cor}\label{thm:isotropicDiffeomorphism}
When $\sigma:\Rn\to\Rn$ is any bounded diffeomorphism, then 
$f\mapsto f\circ\sigma$ is a linear homeomorphism of $F^s_{p,q}(\Rn)$ onto itself for all $s\in\R$.
\end{cor}

This is known from work of Triebel~\cite[Th.~4.3.2]{T92},
which also contains a corresponding result for Besov spaces. (It is this
proof we extended to mixed norms in the previous section.)
The result has also been obtained recently by Scharf~\cite{Sch13}, who covered
all $s\in\R$ by means of an extended notion of atomic decompositions.

In an analogous way,
we also obtain an isotropic counterpart to Theorem~\ref{thm:localVersion}:

\begin{cor}\label{thm:isotropicCompactSupport}
When $\sigma: U\to V$ is a $C^\infty$-bijection between open sets $U,V\subset\Rn$,
then $f\circ\sigma \in F^s_{p,q}(U)$ for every $f\in F^s_{p,q}(V)$ having compact support and 
\begin{equation}
  \|\, f\circ \sigma\, | F^s_{p,q}(U)\| 
  \leq c(\supp f, \sigma) \|\, f\, | F^s_{p,q}(V)\|.
\end{equation}
\end{cor}

\appendix

\section{The Higher-Order Chain Rule}\label{app:higherOrderCR}
For convenience we give a formula for the higher order derivative of a 
composite map
\begin{equation}
  \Rn\xrightarrow{\ f\ } \R^m\xrightarrow{\ g\ }\C.
\end{equation}
Namely, when $f$, $g$ are $C^k$ and  $x_0\in\Rn$, then for every
multi-index $\gamma$ with $1\le |\gamma|\le k$,
\begin{equation}
   \partial^\gamma( g\circ f)(x_0)
  =
\sum_{1\le|\alpha|\le |\gamma|}\partial^{\alpha} g(f(x_0))
  \sum_{\substack{\forall j\colon \alpha_j=\sum n_{\beta^j}\\ \gamma=\sum_{j,\beta^j} n_{\beta^j}\beta^j}}
\gamma!
\prod_{\substack{j=1,\dots,m\\1\le|\beta^j|\le |\gamma|}} \frac1{n_{\beta^j}!}
\Big(\frac{\partial^{\beta^j}f_j(x_0)}{\beta^j!}\Big)^{n_{\beta^j}}.
\label{HCR-eq}
\end{equation}
Hereby the first sum is over multi-indices $\alpha=(\alpha_1,\dots,\alpha_m)$, which in the
second are split
\begin{equation}
  \alpha_1=\sum_{1\le|\beta^1|\le |\gamma|}n_{\beta^1},\ \dots\ , 
  \alpha_m=\sum_{1\le|\beta^m|\le|\gamma|}n_{\beta^m}  
\end{equation}
into integers $n_{\beta^j}\ge0$
(parametrised by $\beta^j=(\beta^j_1,\dots,\beta^j_n)$ in $\N_0^n$, with upper index $j$) 
that fulfil the constraint 
\begin{equation}
  \gamma=\sum_{j=1}^m\sum_{1\le|\beta^j|\le|\gamma|} n_{\beta^j}\beta^j.
\label{gamma-condition}
\end{equation}
Formula \eqref{HCR-eq} and \eqref{gamma-condition} result from Taylor's limit formula:
$g(y+y_0)=\sum_{|\alpha|\le k} c_\alpha y^\alpha +o(|y|^k)$ that holds for $y\to0$
if \emph{and only if}
$c_\alpha=\frac1{\alpha!}\partial^\alpha g(y_0)$ for all $|\alpha|\le k$.
(Necessity is seen recursively for $y\to 0$ along suitable lines; 
sufficiency from the integral remainder.)

Indeed, $k=|\gamma|$ suffices, and with $y=f(x+x_0)-f(x_0)$  Taylor's formula 
applies to both $g$ and to each entry $f_j$ 
(by summing over an auxiliary multi-index $\beta^j\in\N_0^n$),
\begin{equation}
  \begin{split}
    g(f(x+x_0))
  &=\sum_{|\alpha|\le k}\frac{1}{\alpha!}\partial^{\alpha} g(f(x_0))y_1^{\alpha_1}\dots
  y_m^{\alpha_m} +o(|y|^k)
\\
  &=\sum_{|\alpha|\le k}\partial^{\alpha} g(f(x_0))
  \prod_{j=1}^{m} \frac{1}{\alpha_j!}\big(\sum_{1\le |\beta^j|\le k}
    \frac{x^{\beta^j}}{(\beta^j)!}\partial^{\beta^j}\!f_j(x_0)+o(|x|^k)\big)^{\alpha_j}
  +o(|y|^k).
  \end{split}
\label{HCR'-eq}
\end{equation}
Here the first remainder is $o(|x|^k)$ since $o(|y|^k)/|x|^k=o(1)(|f(x+x_0)-f(x_0)|/|x|)^k\to0$.
Using the binomial formula and expanding $\prod_{j=1}^m$, 
the other remainders are also seen to contribute by terms that
are $o(|x|^k)$, or better; whence a single $o(|x|^k)$ suffices.

Hence we shall expand $(\dots)^{\alpha_j}$ using the multinomial formula. 
So we split $\alpha_j=\sum n_{\beta^j}$, with integers $n_{\beta^j}\ge0$ in the sum over all
multi-indices $\beta^j\in\N_0^n$ with $1\le|\beta^j|\le k$.  
The corresponding multinomial coefficient is $\alpha_j!/\prod_{\beta^j} (n_{\beta^j})!$,
so \eqref{HCR'-eq} yields
\begin{equation}
  \begin{split}
    g(f(x+x_0)) &=\sum_{|\alpha|\le k}\partial^{\alpha} g(f(x_0))
  \prod_{j=1}^{m} \sum_{\alpha_j=\sum n_{\beta^j}}
\prod_{1\le|\beta^j|\le k} \frac1{n_{\beta^j}!}
\Big(\frac{x^{\beta^j}}{\beta^j!}\partial^{\beta^j}\! f_j(x_0)\Big)^{n_{\beta^j}}
  +o(|x|^k).
  \end{split}
\label{HCR''-eq}
\end{equation}
Calculating these products, of factors having a choice of
$\alpha_j=\sum n_{\beta^j}$ for each $j=1,\dots,m$, one obtains 
polynomials $x^\omega$ associated to multi-indices $\omega
=\sum_{j=1}^m\sum_{1\le|\beta^j|\le k}n_{\beta^j}\beta^j$.

For $|\omega|>k$ these are $o(|x|^k)$, hence contribute to the
remainder. Thus modified, \eqref{HCR''-eq} is Taylor's formula of order $k$ for $g\circ f$, so 
that $\partial^\gamma(g\circ f)(x_0)/\gamma!$ is given by the
coefficient of $x^\omega$ for $\omega=\gamma$, which yields \eqref{gamma-condition} 
and \eqref{HCR-eq}.

This concise proof has seemingly not been worked out before,
so it should be interesting in its own right. 
E.g.~the Taylor expansions make the presence of the $\beta^j$ obvious,
and the condition $\gamma=\sum_{j,\beta^j} n_{\beta^j}\beta^j$ is natural.
Also the constants $\gamma!/\prod n_{\beta^j}!$ and $(\beta^j)!^{-n_{\beta^j}}$
lead to easy applications.
Clearly $\partial^\alpha g(f(x_0))$ is multiplied by
a polynomial in the derivatives of $f_1$, \dots, $f_m$, which has degree 
$\sum_{j=1}^m\sum_{\beta^j} n_{\beta^j}=\sum_j\alpha_j=|\alpha|$.

The formula \eqref{HCR-eq} itself is well known 
for $n=1=m$ as the Faa di Bruno formula; cf.\ \cite{Jsn02} for its history.
For higher dimensions, the formulas seem to have been less explicit.

The other contributions we know have been rather less straightforward, 
because of reductions, say to $f,g$ being polynomials (or to finite Taylor series), and/or by use
of lengthy combinatorial arguments with recursively given polynomials,
which replace the sum over the $\beta^j$ in \eqref{HCR-eq}; such as the Bell polynomials
that are used in e.g.\ \cite[Thm.~4.2.4]{Rod93}.

Closest to the present approach, we have found the contributions
\cite{Spd05} and \cite{Frae78} in case of one and several variables, respectively.

\providecommand{\bysame}{\leavevmode\hbox to3em{\hrulefill}\thinspace}

\end{document}